\documentclass[a4paper,12pt]{amsart}
\usepackage{amssymb,amsfonts,amsthm}
\usepackage{amstext}
\usepackage{comment}
\usepackage{amsmath}
\usepackage{hyperref}  

\usepackage{color}
\usepackage[latin1]{inputenc}
\usepackage[active]{srcltx}
\usepackage{mathrsfs} 
\usepackage{comment} 
\usepackage{graphicx}
\usepackage{amssymb,amsfonts,amstext, amsmath}
\usepackage[latin1]{inputenc}
\usepackage{tensor}
\usepackage{color}
\usepackage{mathptmx}

\newtheorem{thm}{Theorem}[section]

\newtheorem{lem}[thm]{Lemma}
\newtheorem{prop}[thm]{Proposition}

\newtheorem{defn}[thm]{Definition}

\DeclareMathAlphabet\EuScript{U}{eus}{m}{n}
\SetMathAlphabet\EuScript{bold}{U}{eus}{b}{n}

\newtheorem*{rep@theorem}{\rep@title}
\newcommand{\newreptheorem}[2]{%
	\newenvironment{rep#1}[1]{%
		\def\rep@title{#2 \ref{##1}}%
		\begin{rep@theorem}}%
		{\end{rep@theorem}}}
\makeatother

\DeclareMathOperator{\sech}{sech}
\DeclareMathOperator{\arccosh}{arccosh}

\theoremstyle{definition}

\usepackage{mathrsfs} 

\usepackage{comment} 

\usepackage{hyperref}  
\begin{document}
	
	\title{On Gaussian kernels on Hilbert spaces and kernels on Hyperbolic spaces}
	
	\author{J. C. Guella}
	\email{jean.guella@riken.jp}
	\address{RIKEN Center for Advanced Intelligence Project, Tokyo, Japan}
	
	\begin{abstract}
		This paper describes the concepts of   Universal/ Integrally Strictly Positive Definite/ $C_{0}$-Universal for the Gaussian kernel on a Hilbert space. As a consequence we obtain a similar characterization for an important family of kernels studied and developed by Schoenberg and also on a family of spatial-time kernels popular on geostatistics, the Gneiting class,  and its generalizations. Either by using similar techniques, or by a direct  consequence of the Gaussian kernel on Hilbert spaces, we characterize the same concepts for a family of kernels defined on a Hyperbolic space.   
	\end{abstract}
	

	\maketitle


	\tableofcontents

\section{Introduction}
The concept of a complex valued positive definite kernel has been permeating Mathematics since the beginning of the $20$th century, especially after the seminal work  \cite{Aronszajn1950}, which laid down the connection  between positive definite kernels and Reproducing Kernel Hilbert Spaces (RKHS). In applications (especially in Machine Learning), one of the main desirable properties on a RKHS is if it can approximate a target (but usually unknown) function. In this sense, the concepts of  universality (ability to approximate continuous functions on compact sets) and $C_{0}$-universality (ability to approximate  any $C_{0}$ function) are a basic requirement \cite{CS}, \cite{CZ}.

Schoenberg in \cite{schoenbradial} proved a foundational result in metric geometry, by showing that  a metric space $(X, D)$ can be isometrically embedded into some Hilbert space if and only if the kernel $e^{-tD^{2}(x,y)}$ is  positive definite  for every $t>0$. For instance, spheres  and hyperbolic spaces are not embedable, \cite{gangolli}, \cite{farauthyper}, \cite{feragen2015} . Later, this result was extended to a broader context, and it is usually presented as an equivalent definition for when a kernel $\gamma: X \times X \to \mathbb{C}$ is  conditionally negative definite, by replacing $D^{2}(x,y)$ with $\gamma(x,y)$. One of the most important and widely used positive definite kernels is the Gaussian kernel $G_{\sigma}(x,y)=e^{-\sigma\|x-y\|^{2}}$ ($\sigma >0$) defined on a Euclidean space $\mathbb{R}^{m}$, which is not only universal but in fact can approximate any differentiable function and its derivatives of any order on any compact set simultaneously \cite{simon2018kernel}. 

The major aim of this article is to prove that the Strictly Positive Definite/Universal/Integrally Strictly Positive Definite/ $C_{0}$-Universal are properties that occur not only on the Gaussian kernel but on a larger class among the Schoenberg kernels $e^{-\gamma(x,y)}$ ($\gamma$ is  conditionally negative definite), being  the characterization dependent on somewhat easily verifiable properties of the kernel $\gamma$. These results are presented on Section \ref{Universality of Schoenberg-Gaussian kernels} and are achieved as a corollary of the results on Section \ref{Gaussian kernel on Hilbert spaces}, where we prove that the Gaussian kernel is Strictly Positive Definite/Universal/Integrally Strictly Positive Definite/ $C_{0}$-Universal on any Hilbert space, by using several  versions of the famous Stone-Weierstrass theorem, instead of the standard procedure by using the Fourier transform and its properties.

The Gaussian kernel also served as a building block to generate positive definite kernels on a product of spaces (also called spatio-temporal), being one of the most important examples (especially on geostatistics)  the Gneiting class  \cite{Gneitingclass}, initially proposed as a kernel on $\mathbb{R}^{m^{\prime}}\times \mathbb{R}^{m}$ and recently extended to $X \times \mathbb{R}^{m}$ \cite{GneitingClassmeneoliporc}. Although having its popularity, none qualitative property of this family of kernels has been analyzed on the literature so far. On Section \ref{Kernels on product of spaces} we present a natural generalization of \cite{GneitingClassmeneoliporc} and  provide sufficient conditions for when this generalized family of kernels are Strictly Positive Definite/Universal/Integrally Strictly Positive Definite/ $C_{0}$-Universal.  The proofs are a consequence of the results on Section \ref{Gaussian kernel on Hilbert spaces} together with analysis of when the Schur/Hadamard product of continuous positive definite kernels is Strictly Positive Definite/Universal/Integrally Strictly Positive Definite/ $C_{0}$-Universal given that one of them satisfies this property, presented on Subsection \ref{Products of positive definite kernels}.

We conclude the article on Section \ref{Kernels and hyperbolic spaces}, where the focus is on kernels on Hyperbolic spaces and related types. A special family of positive definite kernels on hyperbolic spaces invariant by the hyperbolic distance, which shares some similarities with completely monotone functions, are analyzed and the concepts of Strictly Positive Definite/Universal/Integrally Strictly Positive Definite/ $C_{0}$-Universal are fully characterized.

\section{Definitions}
 A kernel $K: X \times X \to \mathbb{C}$ is called \textbf{positive definite} if for every finite quantity of distinct points $x_{1}, \ldots, x_{n} \in X$ and scalars $c_{1}, \ldots, c_{n} \in \mathbb{C}$, we have that
$$
\int_{X}\int_{X}K(x,y)d\lambda(x)d\overline{\lambda}(y)=\sum_{\mu, \nu =1}^{n}c_{\mu}\overline{c_{\nu}} K(x_{\mu}, x_{\nu}) \geq 0,
$$
where $\lambda= \sum_{\mu=1}^{n}c_{\mu}\delta_{x_{\mu}}$. In addition, if the above double sum is zero only when all scalars $c_{\mu}$ are zero, we say that the kernel is \textbf{strictly positive definite (SPD)}. The set of discrete measures on $X$ used before are denoted by the symbol $\mathcal{M}_{\delta}(X)$.

The  reproducing kernel Hilbert space (RKHS) of a positive definite kernel $K: X \times X \to \mathbb{C}$ is the Hilbert space $\mathcal{H}_{K} \subset \mathcal{F}(X, \mathbb{C})$, and it satisfies:  $\langle F,  K_{y}  \rangle_{\mathcal{H}_{K}}= F(y)$ for every $F : X \to \mathbb{C}$ that is an element of $\mathcal{H}_{K}$, $ [K_{y}](x):=  K(x,y)$ is an element of $\mathcal{H}_{K}$ for every $y \in X$ and $span\{K_{y}, y \in X\}$ is dense on $\mathcal{H}_{K}$, \cite{Steinwart}.

Recall that for a locally compact space $X$, the Banach space $C_{0}(X)$ is defined as the set of continuous functions $f: X \to \mathbb{C}$ such that for every $\epsilon >0$ there exists a compact set $\mathcal{C}_{\epsilon}$ for which $ |f(x)|< \epsilon$ for   $x \in X \setminus\mathcal{C}_{\epsilon}$, with norm given by $\sup_{x \in X}|f(x)|$.

\begin{defn}\label{univdef} Let $X$ be a Hausdorff space and $K: X \times X \to \mathbb{C} $ be a positive definite kernel. We say that the kernel $K$ is:\\
$\circ$  \textbf{Universal}, if $\mathcal{H}_{K} \subset C(X)$ and for every compact set $\mathcal{C} \subset X$, every continuous function $g:  \mathcal{C} \to \mathbb{C}$ and every $\epsilon >0$ there exists $f : X \to \mathbb{C} \in \mathcal{H}_{K}$ for which
$$	
\sup_{x \in \mathcal{C}} | f(x) - g(x)  |< \epsilon.	 
$$
In addition, when $X$ is a locally compact space, we say that the kernel  $K$ is:\\
$\circ$ \textbf{$C_{0}$-universal}, if $\mathcal{H}_{K}   \subset C_{0}(X)$ and for  every continuous function $g \in C_{0}(X)$ and every $\epsilon >0$ there exists $f : X \to \mathbb{C} \in \mathcal{H}_{K}$ for which
$$	
\sup_{x \in X} | f(x) - g(x)  |< \epsilon.	 
$$
\end{defn}

In other words, a kernel $K:X\times X \to \mathbb{C}$ is universal if its RKHS are made of continuous functions that  when restricted to any compact set $\mathcal{C} \subset X$ are dense on the Banach space $C(\mathcal{C})$. A kernel $K:X\times X \to \mathbb{C}$ is $C_{0}$-universal if its RKHS are made of $C_{0}(X)$ functions  that are dense on the Banach space $C_{0}(X)$.

On the $C_{0}$ case we assume that $X$ is locally compact in order to avoid pathological topologies. In \cite{carmelivitotoigoumanita2010}, it was presented the following criteria for $\mathcal{H}_{K}$ to be a subset of $C(X)$ and  $C_{0}(X)$: 
\begin{prop}\label{rkhscontained}Let $X$ be a Hausdorff space, $K: X \times X \to \mathbb{C}$ be a positive definite kernel. Then:
\begin{enumerate}
	\item[(i)] $\mathcal{H}_{K} \subset C(X)$ if and only if the function $x \in X \to K(x,x) \in \mathbb{C}$ is locally bounded and the function $x \in X \to K(x,y) \in C(X)$, for every  $y \in X$.
		\item[(ii)] $\mathcal{H}_{K} \subset C_{0}(X)$ if and only if the function $x \in X \to K(x,x) \in \mathbb{C}$ is bounded and the function $x \in X \to K(x,y) \in C_{0}(X)$, for every $y \in X$.
\end{enumerate}	
\end{prop}
 
 Although the definition for a positive definite kernel being universal (or $C_{0}-$universal) is simple, it is important to have a condition for these properties when we do not have the description for the  RKHS of a kernel. A direct consequence of  \cite{micchelli2006universal}, a kernel $K: X \times X \to \mathbb{C}$ for which $\mathcal{H}_{K} \subset C(X)$ is universal if and only if the only finite complex valued Radon measure of compact support $\lambda$ on $X$  such that
\begin{equation}\label{univform}
 \int_{X}\int_{X}K(x,y)d\lambda(x) d\overline{\lambda}(y)=0 
\end{equation}
 is the zero measure. We emphasize that the double integral in Equation \ref{univform} is always a nonnegative number because $K$ is positive definite and $\mathcal{H}_{K} \subset C(X)$. In order to simplify the notation, we denote by $\mathcal{M}_{c}(X)$  the set of finite complex valued Radon measures of compact support on a Hausdorff set $X$.
 
 Similarly, by \cite{Sriperumbudur3}    a kernel $K: X \times X \to \mathbb{C}$ for which $\mathcal{H}_{K} \subset C_{0}(X)$ is $C_{0}-$universal if and only if the only finite complex valued Radon measure  $\lambda$ on $X$ such that
\begin{equation}\label{c0univform}
\int_{X}\int_{X}K(x,y)d\lambda(x) d\overline{\lambda}(y)=0 
\end{equation}
is the zero measure. Again, we emphasize that the double integral in Equation \ref{c0univform} is always a nonnegative number because $K$ is positive definite and $\mathcal{H}_{K} \subset C_{0}(X)$. We denote by $\mathcal{M}(X)$  the set of finite complex valued Radon measures  on a  Hausdorff space $X$.

We recall that a  finite Radon measure $\lambda$ on a  Hausdorff space $X$ is a Borel measure for which its total variation $|\lambda|$ is a finite measure and  satisfy
\begin{enumerate}
    \item[(i)](Inner regular)$|\lambda|(E)= \sup\{|\lambda|(K), \ \ K \text{ is compact }, K\subset E\} $ for every Borel set $E$.
    \item[(ii)](Outer regular) $|\lambda|(E)= \inf\{|\lambda|(U), \ \ U \text{ is open }, E\subset U\} $ for every Borel set $E$.
\end{enumerate}
where the outer regularity holds for every measurable set (instead of the usual definition on open sets) because the measure is finite. See section $7$, especially Proposition $7.5$ in \cite{folland} for more details.

 Sometimes,  the inclusion $\mathcal{H}_{K} \subset C_{0}(X)$ is difficult to verify,  but the relation at Equation \ref{c0univform} is much simpler to analyze. 
 
 \begin{defn}\label{intstricdefn} Let $X$ be a  Hausdorff space, we say that a bounded positive definite kernel  $K: X \times X \to \mathbb{C} $ for which $\mathcal{H}_{K}\subset C(X)$ is \textbf{integrally strictly positive definite (ISPD)} if the relation at Equation \ref{c0univform} is satisfied.
 \end{defn}
 
The  definition of an ISPD kernel is based on the one given in \cite{Sriperumbudur}, and can be reinterpreted as $\mathcal{H}_{K}$ being dense on $L^{1}(|\lambda|, X)$ for every nonzero measure $\lambda \in \mathcal{M}(X)$. For some specific type of complex valued kernels, a good description of those who are ISPD were obtained in  \cite{cheney1995}, \cite{Sriperumbudur2}, \cite{Sriperumbudur}, especially  the kernels on Euclidean spaces invariant by translations (more generally on a locally compact commutative group). Usually on the definition of ISPD kernel is  assumed that $X$ is a locally compact space, however it will be convenient for us to use this broader definition since we will be dealing with infinite dimensional Hilbert spaces.

In particular, on an ISPD kernel we can define an inner product on  $\mathcal{M}(X)$ by the formula  
$$
(\mu, \nu) \in \mathcal{M}(X) \times \mathcal{M}(X) \to \int_{X}\int_{X}K(x,y)d\mu(x)d\overline{\nu}(y).
$$
The metric obtained from this inner product is usually denoted as the maximum mean discrepancy (MMD), \cite{gretton2012kernel}.

If the kernel $K$ is real valued, it is sufficient to test the double integrals for real valued measures in $\mathcal{M}(X)$. The concepts of SPD/Universality/ $C_{0}$-Universality/ISPD also exists on the operator valued context \cite{jean2020}, \cite{Caponnetto2008}. Since, we  only use the matrix valued setting, we use the simpler definition that a matrix valued kernel $K:X \times X \to M_{\ell}(C)$ is PD/SPD/Universal/ $C_{0}$-Universal/ISPD if the scalar valued kernel $L:(X \times \{1,\ldots , \ell\}) \times (X \times \{1,\ldots , \ell\})  \to \mathbb{C}$ given by $L((x,i), (y,j))= K_{i,j}(x,y)$ is PD/SPD/Universal/$C_{0}$-Universal/ISPD. 

\section{Gaussian kernel on Hilbert spaces}\label{Gaussian kernel on Hilbert spaces}

Throughout this Section $\mathcal{H}$ denotes a real Hilbert space.

\begin{thm}\label{Gauszinho} The Gaussian kernel $G_{\sigma} : \mathcal{H}\times \mathcal{H} \to \mathbb{R}$, given by  
$$
G_{\sigma}(x,y)= e^{-\sigma\|x-y\|^{2}}
$$
is SPD and universal for every $\sigma >0$.
\end{thm}

 The proof of Theorem \ref{Gauszinho} is based on the famous Stone-Weierstrass Theorem. A similar characterization is possible for the ISPD  where the key argument  is  a   version of the Stone-Weierstrass Theorem for integrable functions proved on \cite{Farrell}. However, on \cite{Farrell} it is an hypothesis that the elements on the algebra of functions are Baire measurable, which is not clear to us if and how this hypothesis can be fulfilled. Being the main ingredient for the proof the inner regularity  on all measurable sets, and  every finite  Radon measure  satisfies this, we could still use the result on our setting. We prove this simple change of \cite{Farrell} at Section \ref{Dense algebras of bounded integrable functions on finite measures}.

\begin{thm}\label{Gauszao}  The Gaussian kernel
$$
(x,y) \in \mathcal{H}\times \mathcal{H} \to G_{\sigma}(x,y) = e^{-\sigma \|x-y\|^{2}} \in \mathbb{R}
$$
is  ISPD.
\end{thm}

If $\mathcal{H}$ is infinite dimensional then it is not a locally compact space, so the concepts of $C_{0}$-universality  are not well defined for  $G_{\sigma}$. However, we can analyse the kernel when restricted to a locally compact space $X \subset \mathcal{H}$ (induced topology).  The following  structure result characterizes when the inclusion $\mathcal{H}_{G_{\sigma}} \subset C_{0}(X)$ is satisfied.

\begin{lem}\label{Gausskerinfty}  Let  $X \subset \mathcal{H}$ be   locally compact and the kernel $G_{\sigma}$ restricted to $X$. The following conditions are equivalent
\begin{enumerate}
    \item[(i)] There exists $z_{0} \in X $ for which the function $G_{\sigma,z_{0}}(x)= e^{-\sigma\|x-z_{0}\|^{2}} \in C_{0}(X)$.
    \item[(ii)] The function $G_{\sigma,z}(x)= e^{-\sigma\|x-z\|^{2}}$ is an element of $C_{0}(X)$ for every $z \in X$.
    \item[(iii)] The inclusion $\mathcal{H}_{G_{\sigma}}\subset C_{0}(X)$ holds.
    \item[(iv)] Every bounded and closed set on $X$ is a compact set on $X$.
\end{enumerate} 
\end{lem}

Next theorem  is a consequence of Theorem  \ref{Gauszao} and Theorem \ref{Gausskerinfty}, however, we present a different proof for it, based on the $C_{0}$ version of the Stone-Weierstrass Theorem.

\begin{thm}\label{gausssunboundc0} Let  $X \subset \mathcal{H}$ be   locally compact. The Gaussian kernel
$$
(x,y) \in X\times X \to G_{\sigma}(x,y):=e^{-\sigma\|x-y\|^{2}} \in \mathbb{R}
$$
is $C_{0}(X)$-universal if and only if $\mathcal{H}_{G_{\sigma}}\subset C_{0}(X)$.
\end{thm}

The results in this section could be proved on a more general setting. By \cite{schoenbradial} a continuous function $g:[0, \infty) \to \mathbb{R}$  is such that the kernel
$$
(x,y ) \in \mathbb{R}^{m} \times \mathbb{R}^{m} \to g(\|x-y\|) \in \mathbb{R}
$$
is positive definite for every $m \in \mathbb{N}$, if and only if $f(t):= g(\sqrt{t}) \in C^{\infty}((0, \infty))$ with $(-1)^{n}f^{(n)}(t) \geq 0$ for every $n \in \mathbb{N}$ (a function $f$ with these properties is called  \textbf{completely monotone}), or equivalently that there exists a nonnegative measure $\lambda \in \mathcal{M}([0, \infty))$ for which
$$
g(t) = \int_{[0, \infty)}e^{-rt^{2}}d\lambda(r).
$$
Replacing the Gaussian kernel by a function of this type on Theorems \ref{Gauszinho}, \ref{Gauszao} and \ref{gausssunboundc0} is possible, whenever $g$ is not a constant function, or equivalently $\lambda((0, \infty))>0$.  Lemma \ref{Gausskerinfty}  is also possible  whenever $g$ is not a constant function and relation $(iv)$ is replaced by
\begin{enumerate}
\item[$(iv)^{\prime}$] Every bounded and closed set on X is a compact set and $\lim_{t \to \infty}g(t)=0$.
\end{enumerate}

The argument that this generalization is indeed possible is a direct consequence of Theorem $3.7$ in \cite{jean2020}, and we do not present it.


\section{Universality of Schoenberg-Gaussian kernels}\label{Universality of Schoenberg-Gaussian kernels}
In \cite{schoenmetric}  Schoenberg proved that a kernel $\gamma: X \times X \to \mathbb{R}$ is such that the kernel
$$
(x,y) \in X \times X \to e^{-r\gamma(x,y)} \in \mathbb{R}
$$
is positive definite for every $r >0$ if and only if the kernel $\gamma$ is \textbf{conditionally negative definite} (CND), that is, $\gamma$ is symmetric  ($\gamma(x,y)= \gamma(y,x)$) and for every finite quantity of distinct points $x_{1}, \ldots, x_{n}$ and scalars $c_{1}, \ldots, c_{n} \in \mathbb{R}$, restricted to the hyperplane $\sum_{\mu=1}^{n} c_{\mu}=0$, it satisfies
$$
\sum_{\mu, \nu=1}^{n}c_{\mu}c_{\nu}\gamma(x_{\mu}, x_{\nu}) \leq 0.
$$
Since \cite{Aronszajn1950} it is known the  strong connection between positive definite kernels and inner products on Hilbert spaces as well as conditionally negative definite kernels and norms on Hilbert spaces, since   $\gamma: X \times X \to \mathbb{R}$ can be written as (Proposition $3.2$ in \cite{berg0})
\begin{equation}\label{condequa}
\gamma(x,y)= \|h(x)- h(y)\|_{\mathcal{H}}^{2} + f(x) + f(y)
\end{equation}
where $\mathcal{H}$ is a real Hilbert space and $h: X \to \mathcal{H}$, and $f : X \to \mathbb{R}$. Note that $f(x)= \gamma(x,x)/2$. This description allows us to understand the kernel $e^{-\gamma(x,y)}$ as a weighted version ($f$ may be nonzero) of a restriction of the  Gaussian kernel defined on  an (usually) infinite dimensional Hilbert space.

An important relation to our purposes is if the function $h$ is injective (equivalently, if  $2\gamma(x,y) > \gamma(x,x) + \gamma(y,y)$ for every $x,y \in X$). On this case there  is a natural metric structure on $X$ provided by the norm on $\mathcal{H}$, being the distance
$$
D_{\gamma}(x,y) := \sqrt{ \gamma(x,y)- \frac{\gamma(x,x)}{2} + \frac{\gamma(y,y)}{2} }.
$$
 Naturally,  a conditionally negative definite kernel with this property is called \textbf{metrizable}. The set $X$ with the metric topology $D_{\gamma}$ is denoted as $X_{\gamma}$.

\begin{thm}\label{improvSchoen} Let $X$ be a  Hausdorff space and $\gamma: X \times X \to \mathbb{R}$ be a  continuous conditionally negative definite kernel. The kernel
$$	
(x,y) \in X \times X \to G_{\gamma}(x,y):= e^{-\gamma(x,y)} \in \mathbb{R}	
$$	
is SPD (universal) if and only if the kernel $\gamma$ is metrizable.
\end{thm}


\begin{thm}\label{integrallygaussianonmetric} Let $X$ be a  Hausdorff space and $\gamma: X \times X \to \mathbb{R}$ be a   continuous conditionally negative definite   kernel. Then the kernel
$$
(x,y) \in X\times X \to e^{-\gamma(x,y)} \in \mathbb{R}
$$
is  ISPD if and only if $\gamma$ is metrizable and the function $x \in X \to \gamma(x,x) \in \mathbb{R}$ is bounded from below.
\end{thm}

 The topologies of $X$ and $X_{\gamma}$ might be equivalent, for instance if $g: [0, \infty) \to \mathbb{R}$ is a continuous function for which $\lim_{t\to \infty}g(t) \in (0, \infty)\cup \{\infty\}$, $g(0)=0$, $g(t) \in (0, \infty)$ for $t \in (0, \infty)$  and such that the radial kernel $(x,y) \in \mathbb{R}^{m}\times \mathbb{R}^{m} \to g(\|x-y\|) $ is conditionally negative definite, then the metric generated from this kernel on $\mathbb{R}^{m}$ is equivalent to the Euclidean metric on $\mathbb{R}^{m}$.
 
 The following  structure result elucidates some aspects concerning the inclusion $\mathcal{H}_{G_{\gamma}} \subset C_{0}(X)$.

 \begin{lem}\label{condnegkerinfty}  Let $X$ be a locally compact Hausdorff  space and  $\gamma: X \times X \to \mathbb{R}$ be a  continuous  conditionally negative definite metrizable kernel for which the function $x \in X \to \gamma(x,x) \in \mathbb{R}$ is bounded. Then $\mathcal{H}_{G_{\gamma}} \subset C_{0}(X)$  if and only if  there exists $z_{0} \in X $ for which the function $G_{\gamma,z_{0}}(x)= e^{-\gamma(x,z_{0})}$ is an element of $C_{0}(X)$.
 \end{lem}



\begin{thm}\label{gaussunboundc0} Let $X$ be a locally compact Hausdorff space and $\gamma: X \times X \to \mathbb{R}$ be a continuous conditionally negative definite kernel. The kernel
$$
(x,y) \in X\times X \to e^{-\gamma(x,y)} \in \mathbb{R}
$$
is $C_{0}(X)$-universal if and only if $\gamma$ is metrizable and $\mathcal{H}_{G_{\gamma}} \subset C_{0}(X)$.
\end{thm}

\begin{thm}\label{gaussispd} Let $f:[0, \infty) \to \mathbb{R}$ be a non constant completely monotone function  and $\gamma: X \times X \to [0, \infty)$ be a continuous conditionally negative definite metrizable kernel, then the kernel
$$
(x,y) \in X\times X \to f(\gamma(x,y)) \in \mathbb{R}
$$
is ISPD.
\end{thm}

\section{Gneiting class and related kernels}\label{Kernels on product of spaces}
Based on the results of the previous Sections we  are able to prove qualitative properties of some important generalizations of the Gaussian (and related) kernel to a product of spaces. 

\subsection{Gneiting class}
A popular example, especially on geostatistics, of such kernels is the \textbf{Gneiting class}  \cite{Gneitingclass}, initially proposed as the family of positive definite kernels
$$
((u,x),(v,y)) \in ( \mathbb{R}^{m^{\prime}}\times \mathbb{R}^{m} )^{2} \to g(\|u-v\|^{2})^{-m/2}\psi \left (\frac{\|x-y\|^{2}}{g(\|u-v\|^{2})}\right ) \in \mathbb{R}   
$$
where $g,\psi : [0, \infty) \to \mathbb{R}$ are continuous and nonconstant functions, $g$ is a positive function, $\psi$ is completely monotone and $g$ is a Bernstein function, that is $g \in C^{\infty}((0, \infty))$ and $g^{\prime}$ is completely monotone. Several extensions and applications of this type of kernel have been proposed and proved \cite{porcu30}, \cite{GneitingClassmeneoliporc}. We focus on a generalization that encloses all of the above mentioned.  

Let $X$ be a  Hausdorff space, $\gamma: X \times X \to (0, \infty)$ be a  continuous conditionally negative definite kernel and $A : X \times X \to \mathbb{C}$ be a continuous  kernel. Suppose that the kernel 
	$$
	(u,v) \in X \times X \to C(u,v) := A(u,v)\gamma(u,v)^{m/2} \in \mathbb{R}
	$$
is positive definite. Under this hypothesis we define the kernel $G_{A, \gamma}: (X \times \mathbb{R}^{m}) \times (X \times \mathbb{R}^{m}) \to \mathbb{C}$ as
$$
 G_{A, \gamma}((u,x),(v,y)) := A(u,v)e^{-\|x-y\|^{2}/\gamma(u,v)} 
$$

\begin{thm}\label{genGaussianuniversal} The kernel $G_{A, \gamma}$ is positive definite and continuous. If  $\gamma$ is a metrizable kernel, then $G_{A, \gamma}$ is  SPD (or universal) if and only if $A(u,u)>0$ for every $u \in X$.  
\end{thm}	   

 As a consequence that the functions fulfilling Bochner's Theorem are uniquely representable,  the hypothesis that the kernel $C$ is positive definite is in fact a necessary condition for $G_{A, \gamma}$ be positive definite. Also, although we are not imposing that the kernel $A$ is positive definite, it is positive definite because the kernel $C$ is positive definite and $\gamma^{-m/2}(u,v)$ as well by Lemma \ref{gammakernel}. 


 \begin{thm}\label{genGaussianintegralllymetric}  If the kernel $\gamma$ is metrizable, the function $x \in X \to \gamma(x,x) \in \mathbb{R}$ is bounded from below and $C$ is bounded, then the kernel $G_{A, \gamma}$ is ISPD if and only if  if and only if  $A(u,u)>0$ for every $u \in X$. \\
 If $X$ is a locally compact space, the inclusion $\mathcal{H}_{G_{A, \gamma}}\subset C_{0}(X\times \mathbb{R}^{m})$ occurs if and only if $\mathcal{H}_{A} \subset C_{0}(X)$.
\end{thm}

In particular, by using the previous theorem  and Theorem $3.7$ in \cite{jean2020}  a kernel among the initially proposed Gneiting class in \cite{Gneitingclass} is $C_{0}$-universal if and only if the function $g$ is unbounded and $\lim_{t \to \infty} \psi(t)=0$.  

When $X$ is a finite set (on which the kernel $G_{A, \gamma}$ can be understood as a matrix valued kernel on $\mathbb{R}^{m}$), it is possible to characterize when $G_{A, \gamma}$ is universal/$C_{0}$-universal (when $X$ is finite we always have that $\mathcal{H}_{A} \subset C_{0}(X)$)  even if  $\gamma$ is not metrizable. It is not clear if on the general setting of Theorem \ref{genGaussianuniversal} the same approach is possible.

\begin{thm}\label{secondCaponnetogeneralization1M22} Let $m, \ell \in \mathbb{N}$ and the matrix valued kernel $G_{A, \gamma}: \mathbb{R}^{m} \times \mathbb{R}^{m} \to M_{\ell}(\mathbb{C})$  given by
	$$
	[G_{A, \gamma}(x,y)]_{\mu, \nu} :=a_{\mu, \nu}e^{-\|x-y\|^{2}/\gamma_{\mu, \nu }} , \quad a_{\mu, \nu} \in \mathbb{C}, \quad  \gamma_{\mu, \nu } >0. 
	$$
	Assume that the matrix $C:=[a_{\mu, \nu}\gamma_{\mu, \nu}^{m/2}]_{\mu, \nu =1}^{\ell}\in M_{\ell}(\mathbb{C})$ is positive semidefinite and the matrix $\Gamma= [\gamma_{\mu, \nu}]\in M_{\ell}(\mathbb{R})$ is conditionally negative definite, then 
	\begin{enumerate}
	\item[(i)] The kernel $K$ is strictly positive definite (universal)  if and only if the matrix $A$ is positive definite.
		\item[(ii)]The kernel $K$ is $C_{0}$-universal  if and only if  for every $F \subset \{1,\ldots, \ell\}$ for which $2\gamma_{\mu,\nu} = \gamma_{\mu, \mu} + \gamma_{\nu, \nu}$ for every $\mu, \nu \in F$, the matrix $C_{F}:= [a_{\mu, \nu}\gamma_{\mu, \nu}^{m/2}]_{\mu, \nu \in F} \in M_{|F|}(\mathbb{C})$ is positive definite.
	\end{enumerate}
\end{thm}



We conclude this subsection with  an interesting interaction between the RKHS of different Gaussian kernels. Define 
$$
H_{G_{\sigma}}:= span\{e^{-\sigma\| x- \cdot\|^{2}}, x \in \mathbb{R}^{m}\} \subset C_{0}(\mathbb{R}^{m}),
$$
$$
H^{c}_{G_{\sigma}}:= \{ y \in \mathbb{R}^{m} \to \int_{\mathbb{R}^{m}}e^{-\sigma\| x- y\|^{2}}d\lambda(x) \in \mathbb{R},  \quad \lambda \in \mathcal{M}_{c}(\mathbb{R}^{m}) \}\subset C_{0}(\mathbb{R}^{m}).
$$ 

\begin{lem}\label{interaction} Let $0 < \sigma_{1}< \ldots< \sigma_{\ell}$, then  
$$
H_{G_{\sigma_{1}}}+ \ldots + H_{G_{\sigma_{\ell}}} \quad \text{ and } \quad 
H^{c}_{G_{\sigma_{1}}}+ \ldots + H^{c}_{G_{\sigma_{\ell}}}
$$ 
are  direct sums.
\end{lem}
Unfortunately, we do not know if Lemma \ref{interaction} is valid on an  infinite dimensional Hilbert space or when $\|x-y\|^{2}$ is replaced by a conditionally negative definite kernel $\gamma$.

 By the arguments presented at the proof of Theorem \ref{secondCaponnetogeneralization1M22}, the same relation does not occur for the space 
$$
H^{1}_{G_{\sigma}}:= \{y \in \mathbb{R}^{m} \to \int_{\mathbb{R}^{m}}e^{-\sigma\| x- y\|^{2}}d\lambda(x) \in \mathbb{R}, \quad \lambda \in \mathcal{M}(\mathbb{R}^{m}) \} \subset C_{0}(\mathbb{R}^{m}).
$$
Note that $H_{G_{\sigma}} \subset H^{c}_{G_{\sigma}} \subset H^{1}_{G_{\sigma}} \subset  \mathcal{H}_{G_{\sigma}}$.

\subsection{Matern family}

Another important set  of kernels on geostatistics is the Matern family 
$$
(x,y) \in \mathbb{R}^{m}\times \mathbb{R}^{m} \to \mathscr{M}(\|x-y\|; \alpha, \nu):= \int_{(0, \infty)}e^{-\|x-y\|^{2}t} \left (  \left(\frac{\alpha^{2}}{4}\right )^{\nu}\frac{t^{-1-\nu}}{\Gamma(\nu)}e^{-\alpha^{2}/4t}\right )dt
$$
which are positive definite for every $m \in \mathbb{N}$, $\mathscr{M}(0; \alpha, \nu)=1$ and also satisfies the equality $\mathcal{M}(\|x\|; \alpha , \nu)=2^{1-\nu}(\|x\|\alpha)^{\nu}\mathcal{K}_{\nu}(\|x\|\alpha)/\Gamma(\nu) $, where $\mathcal{K}_{\nu}$ denotes the modified Bessel
function of the second kind of order $\nu$ \cite{NIST:DLMF}.  A matrix valued version of this family was proposed in  \cite{Gneitingmatern2010} and later was  generalized in \cite{Porcumaterngaussfamily} as the family of matrix valued kernels $C_{i,j}^{\mathscr{M}, \psi}: (\mathbb{R}^{m^{\prime}} \times \mathbb{R}^{m} )^{2}  \to M_{\ell}(\mathbb{C})$ given by 
$$
C_{i,j}^{\mathscr{M}, \psi}((u,x),(v,y)):=c_{i,j}\frac{1}{\psi(\|u-v\|^{2})^{m/2}}\mathscr{M}\left (\frac{\|x-y\|}{\psi(\|u-v\|^{2})^{1/2}} ;\alpha_{i,j}, \nu_{i,j}\right ),
$$
 where the function $\psi:[0, \infty) \to \mathbb{R}$ is a positive Bernstein function, $\alpha_{i,j}= ((\alpha_{i}^{2}+\alpha_{j}^{2})/2)^{1/2}$, $\nu_{i,j}= \nu_{i} +\nu_{j}$, with $\alpha_{i},\nu_{i} \in (0, \infty)$ and the matrix
$$
\left [c_{i,j}\frac{2^{-\nu_{i}}\Gamma(2\nu_{i})^{1/2}}{\alpha_{i}^{\nu_{i}}}\frac{2^{ - \nu_{j}}\Gamma(2\nu_{j})^{1/2}}{ \alpha_{j}^{\nu_{j}}}\frac{(\alpha_{i}+\alpha_{j})^{\nu_{i}+ \nu_{j}}}{\Gamma(\nu_{i}+\nu_{j})} \right]_{i,j=1}^{\ell}
$$
 is assumed to be  positive semidefinite.

Similar to the definition of $G_{A, \gamma}$, let $X$ be a Hausdorff space, $\gamma: X \times X \to (0, \infty)$ be a continuous conditionally negative definite kernel and $A: X \times X \to M_{\ell}(\mathbb{C})$ be a continuous matrix valued kernel. Suppose that the matrix valued kernel  $C: X \times X \to M_{\ell}(\mathbb{C})$ defined as 
$$
 C_{i,j}(u,v):= A_{i,j}(u,v)\gamma^{m/2}(u,v)\frac{2^{-\nu_{i}}\Gamma(2\nu_{i})^{1/2}}{\alpha_{i}^{\nu_{i}}}\frac{2^{ - \nu_{j}}\Gamma(2\nu_{j})^{1/2}}{ \alpha_{j}^{\nu_{j}}}\frac{(\alpha_{i}+\alpha_{j})^{\nu_{i}+ \nu_{j}}}{\Gamma(\nu_{i}+\nu_{j})}
$$
is positive definite. Under these hypothesis we define the kernel $C^{A, \gamma}: ( X \times \mathbb{R}^{m})^{2}  \to M_{\ell}(\mathbb{C})$, by 
$$
C_{i,j}^{A, \gamma}((u,x),(v,y)):=A_{i,j}(u,v)\mathscr{M}\left (\frac{\|x-y\|}{\gamma(u,v)^{1/2}} ;\alpha_{i,j}, \nu_{i,j}\right ),
$$
 where  $\alpha_{i,j}= ((\alpha_{i}^{2}+\alpha_{j}^{2})/2)^{1/2}$, $\nu_{i,j}= \nu_{i} +\nu_{j}$, with $\alpha_{i},\nu_{i} \in (0, \infty)$.

\begin{thm}\label{genmaternuniversal} The matrix valued kernel $C^{A, \gamma}$ is positive definite and continuous.\\
If  $\gamma$ is a metrizable kernel, then $C^{A, \gamma}$ is  SPD (universal) if and only if $A_{i,i}(u,u)>0$ for every $1\leq i \leq \ell$, $u \in X$ and $\{(i,j), \quad (\alpha_{i}, \nu_{i} )= ( \alpha_{j}, \nu_{j})\}= \{(i,i), \quad 1\leq i \leq \ell \}$.   
\end{thm}

\begin{thm}\label{genmaternc0universal}
If the kernel $\gamma$ is metrizable, the function $x \in X \to \gamma(x,x) \in \mathbb{R}$ is bounded from below and $C$ is bounded, then the matrix valued kernel $C^{A, \gamma}$
 is ISPD if and only if  $A_{i,i}(u,u)>0$ for every $u \in X$, $1\leq i \leq \ell$ and $\{(i,j), \quad (\alpha_{i}, \nu_{i} )= ( \alpha_{j}, \nu_{j})\}= \{(i,i), \quad 1\leq i \leq \ell \}$.   .\\
If $X$ is a locally compact space, the inclusion $\mathcal{H}_{C^{A, \gamma}} \subset C_{0}(X \times \mathbb{R}^{m}, \mathbb{C}^{\ell})$ occurs if and only if $\mathcal{H}_{A}\subset C_{0}(X, \mathbb{C}^{\ell})$. \end{thm}

The  definition of the matrix valued kernels $C^{A, \gamma}$ is inspired on the Gneiting class, which turns out to be  well defined only on  Euclidean spaces of a bounded dimension. Being so, this definition does not take advantage that the Matern family is positive definite on all Euclidean spaces. 

To surpass this problem, we define the matrix valued kernel $\mathscr{M}_{A,\gamma} :(  X \times \mathcal{H})\times (X \times \mathcal{H}) \to  M_{\ell}(\mathbb{C})$ as
$$
[\mathscr{M}_{A,\gamma}((x,u),(y,v))]_{i,j}:=A_{i,j}(u,v)\mathscr{M}(\|x-y\|;\gamma(u,v)^{1/2}, \nu_{i} + \nu_{j})  
$$
where $\gamma:X \times X \to (0, \infty)$ is a continuous conditionally negative definite kernel,  $A: X \times X \to M_{\ell}(\mathbb{C})$ is a continuous matrix valued kernel, $\nu_{i}>0$ for every $i$ and under the restriction that the matrix valued kernel 
$C: X \times X \to M_{\ell}(\mathbb{C})$ defined as 
$$
 C_{i,j}(u,v):= A_{i,j}(u,v)\frac{\gamma(u,v)^{\nu_{i} + \nu_{j}}}{\Gamma(\nu_{i} + \nu_{j})}
$$
is positive definite.

\begin{thm}\label{genmaterninfinitouniversal} The following properties  holds. 
\begin{enumerate}
\item[(i)] The matrix valued kernel $\mathscr{M}_{A,\gamma}$ is positive definite and continuous.
\item[(ii)] If  $\gamma$ is a metrizable kernel, then $[\mathscr{M}_{A,\gamma}]_{i,j=1}^{\ell}$ is  SPD (universal) if and only if the positive numbers $\nu_{i}$ are distinct and  $A_{i,i}(u,u)>0$ for every $1\leq i \leq \ell$ and $u \in X$.
\item[(iii)] If the kernel $\gamma$ is metrizable, the function $x \in X \to \gamma(x,x) \in \mathbb{R}$ is bounded from below and $C$ is bounded, then the the matrix valued kernel $\mathscr{M}_{A,\gamma}$ is ISPD  if and only if the positive numbers $\nu_{i}$ are distinct and  $A_{i,i}(u,u)>0$ for every $1\leq i \leq \ell$ and $u \in X$.
\item[(iv)]When $X$  is a locally compact space and $\mathcal{H}= \mathbb{R}^{m}$,  the inclusion $\mathcal{H}_{\mathscr{M}_{A,\gamma}} \subset C_{0}(X\times \mathbb{R}^{m}, \mathbb{C}^{\ell}  )$ occurs if and only if $\mathcal{H}_{A} \subset C_{0}(X, \mathbb{C}^{\ell} )$. 
\end{enumerate}
\end{thm}

\section{Kernels and hyperbolic spaces}\label{Kernels and hyperbolic spaces}
\subsection{Isotropic kernels on real hyperbolic spaces}\label{Isotropic kernels on real hyperbolic spaces}
Let $\mathbb{H}^{m}:=\{(x,t_{x}) \in \mathbb{R}^{m}\times (0,\infty), \quad   t_{x}^{2} - \|x\|^{2} =1 \}$ be the $m$-dimensional real hyperbolic space and consider the kernel
$$
((x,t_{x}),(y,t_{y})) \in \mathbb{H}^{m} \times \mathbb{H}^{m} \to [(x,t_{x}),(y,t_{y})]:= t_{x}t_{y} - \langle x,y\rangle \in [1,\infty),
$$
which satisfies the relation  
$$
\cosh(d((x,t_{x}),(y,t_{y}))) = [(x,t_{x}),(y,t_{y})].
$$
Where $d$ is the geodesic distance in $\mathbb{H}^{m}$. A kernel $K : \mathbb{H}^{m} \times \mathbb{H}^{m} \to \mathbb{R} $ is called \textbf{isotropic} if its invariant by the group 
$$
O(m,1):=\{A \in M_{m+1}(\mathbb{R}), \quad AJ A^{t}= J, \quad  det(A)=1, \quad   A(e_{m+1})= e_{m+1}  \},$$
where $J= Diag(-1, \ldots, -1,1) $, that is, $K(x,y)= K(Ax, Ay)$ for every $x,y \in \mathbb{H}^{m}$ and $A \in O(m,1)$. Similar to isotropic kernels on real spheres \cite{dai}, if $K$ is an isotropic kernel on $\mathbb{H}^{m}$ there exists functions $f: [0, \infty) \to \mathbb{R}$, $g: [1, \infty) \to \mathbb{R}$ for which
$$
K(x,y)= f(d(x,y))= g([x,y]), \quad x, y \in \mathbb{H}^{m}.
$$
 At Section $8$ of \cite{farauthyper} (also page $174$ of \cite{berg0}), it is proved that if $g: [1, \infty) \to \mathbb{R}$ is a continuous function, the kernel
$$
((x,t_{x}),(y,t_{y})) \in \mathbb{H}^{m} \times \mathbb{H}^{m} \to g([(x,t_{x}),(y,t_{y})]) \in \mathbb{R}
$$
is positive definite for every $m \in \mathbb{N}$ if and only if the function $s \in [0,\infty) \to g(e^{s}) \in \mathbb{R}$ is completely monotone, or equivalently that there exists a nonnegative finite measure $\lambda \in \mathcal{M}([0, \infty))$ for which
$$
g(s)=\int_{[0, \infty)}s^{-r}d\lambda(r), \quad s \in [1,\infty). 
$$
In terms of the function $f$, the expression is
$$
f(t)= \int_{[0, \infty)}\sech(t)^{r}d\lambda(r), \quad t \in [0, \infty). $$
In this subsection, we prove several qualitative properties for these kernels in a similar way as Section \ref{Gaussian kernel on Hilbert spaces}. In this sense, a real hyperbolic space is a set $\mathbb{H}$ for which there exists a Hilbert space $\mathcal{H}$ such that
$$
\mathbb{H}=\{(x, t_{x}) \in \mathcal{H}\times (0, \infty), \quad t^{2}_{x} -\|x\|^{2}=1   \}. 
$$
The bilinear form $[\cdot, \cdot]$ is defined analogously, and $d(\cdot , \cdot)=\arccosh([\cdot, \cdot])$ defines a metric on $\mathbb{H}$. 


The following Theorem is a version of Theorem \ref{Gauszinho} and  Theorem \ref{Gauszao} to the hyperbolic setting.

\begin{thm}\label{infihypersecondstep}  The kernel 
$$
(z,w) \in \mathbb{H}\times \mathbb{H} \to H_{r}(z,w):=[z,w]^{-r} = \sech(d(z,w))^{r}  \in \mathbb{R}
$$
is ISPD for every  $r>0$.
\end{thm}

 A similar  structure result as the one in Lemma \ref{Gausskerinfty} characterizes when the inclusion $\mathcal{H}_{H_{r}} \subset C_{0}(X)$ is satisfied.

\begin{lem}\label{hypGausskerinfty}  Let $\mathbb{H}$ be a real hyperbolic space, $X \subset \mathbb{H}$ be   locally compact and $r>0$. The following conditions are equivalent
\begin{enumerate}
    \item[(i)] There exists $\xi_{0} \in X $ for which the function $H_{r,\xi_{0}}(z)= [z,\xi_{0}]^{-r}$ is an element of $C_{0}(X)$.
    \item[(ii)] The function $H_{r,\xi}(z)= [z,\xi]^{-r}$ is an element of $C_{0}(X)$ for every $\xi \in X$.
    \item[(iii)] The inclusion $\mathcal{H}_{H_{r}}\subset C_{0}(X)$ holds.
    \item[(iv)] Every bounded and closed set on $X$ is a compact set.
\end{enumerate} 
\end{lem}

Similar to the comments made at the end of Section \ref{Gaussian kernel on Hilbert spaces}, if   $g: [1, \infty) \to \mathbb{R}$ is a continuous function for which  the kernel
$$
(x,y) \in \mathbb{H}^{m} \times \mathbb{H}^{m} \to g([x,y]) \in \mathbb{R}
$$
is positive definite for every $m \in \mathbb{N}$, replacing the kernel $[\cdot, \cdot ]^{-r}$ on Theorem \ref{infihypersecondstep} by the kernel $g([\cdot, \cdot ])$ is possible, whenever $g$ is not a constant function. Lemma \ref{hypGausskerinfty}  is also possible whenever $g$ is not a constant function and relation $(iv)$ is replaced by
\begin{enumerate}
    \item[$(iv)^{\prime}$] Every bounded and closed set on $X$ is a compact set and $\lim_{s \to \infty}g(s)=0$. 
\end{enumerate}
The argument that this generalization is possible is also a direct consequence of Theorem $3.7$ in \cite{jean2020}, and we do not present it.

\subsection{Hyperbolic  and log-conditional  kernels}\label{Hyperbolic  and log-conditional  kernels}
 A kernel $\beta : X \times X \to \mathbb{R}$ is called  \textbf{hyperbolic} if there exists a real hyperbolic space $\mathbb{H}$ and a function $h : X \to \mathbb{H}$ for which $\beta(x,y)= [h(x),h(y)]$. At \cite{MonodMoebius} it is proved that a kernel $\beta$ is hyperbolic if and only if $\beta(x,x)=1$ for all $x \in X$ and the kernel
\begin{equation}\label{hypimer}
(x,y) \in X\times X \to \beta(x,z)\beta(y,z) - \beta(x,y) \in \mathbb{R}
\end{equation}
is positive definite for some $z \in X$ (or equivalently, for every $z \in X$).

For example, if $\gamma: X \times X \to \mathbb{R}$ is a conditionally negative definite kernel for which $\gamma(x,x)=0$ for every $x \in X$, then the kernel $\beta(x,y):= 1 + \gamma(x,y)$ is hyperbolic, being a possible argument a verification that the kernel on  Equation \ref{hypimer} is positive definite  using the  representation \ref{condequa}. Also, the kernel $$
(x,y) \in \mathcal{H} \times \mathcal{H} \to \sqrt{1+\|x\|^{2}}\sqrt{1 +\|y\|^{2}} - \langle x,y \rangle \in \mathbb{R}
$$ 
is hyperbolic on every Hilbert space $\mathcal{H}$.

The relation between hyperbolic kernels and the functions $s \in [1, \infty) \to s^{-r} \in \mathbb{R}$, $r \in (0, \infty)$,  is different from the relation between conditionally negative definite kernels and the functions $s \in [0, \infty) \to e^{-sr} \in \mathbb{R}$, $r \in (0, \infty)$. 

 If $\gamma: X \times X \to \mathbb{R}$ is conditionally negative definite, by Schoenberg the kernel  $e^{-r\gamma(x,y)}$ is  positive definite  and the kernel $r\gamma(x,y)$ is conditionally negative definite for every $r>0$. However, if $\beta: X \times X \to \mathbb{R}$ is a hyperbolic kernel, by Faraut and Harzallah the kernel $\beta(x,y)^{-r}$ is  positive definite  for every $r>0$, but the  kernel $\beta^{r}(x,y)$ is (with certainty) hyperbolic only  for  $1\geq r>0$, \cite{MonodMoebius}. What occurs is that by $\beta$ being hyperbolic, $\log(\beta(x,y))$ is a conditionally negative definite kernel, and by Schoenberg this property is equivalent to the kernel $e^{-r\log(\beta(x,y))}= \beta(x,y)^{-r}$ being positive definite for every $r>0$. We say that a symmetric  kernel $L: X \times X \to [1,\infty)$ is \textbf{log-conditional} if the kernel $\log L(x,y)$ is conditionally negative definite. Note that if $L$ is log-conditional then so is $L^{r}$ for every $r>0$. 
 
 Being so, a natural question is to analyse when the kernel $L(x,y)^{-r}$  is SPD/Universal/ ISPD/ $C_{0}$-universal, in a similar way as Section \ref{Universality of Schoenberg-Gaussian kernels}. However, since 
 $$
 L(x,y)^{-r} = e^{-r\log L(x,y)}
 $$
such characterizations are a consequence of the results proved on Section \ref{Universality of Schoenberg-Gaussian kernels}. This also includes the results from Subsection \ref{Isotropic kernels on real hyperbolic spaces}.  For completion, we state these characterizations. Naturally, a log-conditional kernel $L$ is metrizable if $\log L$ is  metrizable.

\begin{thm}\label{hyperimprovSchoen} Let $X$ be a  Hausdorff space and $L: X \times X \to [1, \infty)$ be a  continuous log-conditional kernel. Then the kernel
$$	
(x,y) \in X \times X \to H_{L}(x,y):= L(x,y)^{-1} \in \mathbb{R}	
$$	
is SPD (universal) if and only if the kernel $L$ is metrizable.
\end{thm}

Similar to Section \ref{Universality of Schoenberg-Gaussian kernels}, the set $X$ with the metric induced by the conditionally negative definite kernel $\log L(x,y)$ is being denoted by $X_{\log L}$.  If $L$ is a metrizable hyperbolic kernel, the hyperbolic metric $d_{\mathbb{H}}(x,y):=\arccosh L(x,y)$ and the Hilbertian metric $d_{\mathcal{H}}(x,y):=\sqrt{\log L(x,y)} $ are equivalent because  
$$
d_{\mathcal{H}} = \sqrt{\log \cosh d_{\mathbb{H}}}, \quad d_{\mathbb{H}} = \arccosh (e^{(d_{\mathcal{H}} 
)^{2}})
$$
and the functions $\sqrt{\log \cosh t}$, $\arccosh (e^{t^{2}})$ are continuous  on the the interval $[0, \infty)$.

\begin{thm}\label{integrallyhyperonmetric} Let  $L: X \times X \to [1, \infty)$ be a  continuous log-conditional kernel. Then the kernel
$$
(x,y) \in X \times X \to H_{L}(x,y):= L(x,y)^{-1} \in \mathbb{R}
$$
is  ISPD if and only if $L$ is metrizable.
\end{thm}

\begin{thm}\label{hyperkerc0infty}  Let $X$ be a  locally compact Hausdorff space and $L: X \times X \to [1, \infty)$ be a  continuous log-conditional kernel. The kernel
$$
(x,y) \in X \times X \to H_{L}(x,y):= L(x,y)^{-1} \in \mathbb{R}
$$
is $C_{0}(X)$-universal if and only if $L$ is metrizable and $\mathcal{H}_{H_{L}} \subset C_{0}(X)$.\\
Further, there exists $z_{0} \in X $ for which the function $H_{L,z_{0}}(x)= L(x,z_{0})^{-1}$ is an element of $C_{0}(X)$ if and only if   $\mathcal{H}_{H_{L}} \subset C_{0}(X)$
\end{thm}

\newpage
\section{Proofs}
\subsection{Section \ref{Gaussian kernel on Hilbert spaces}}

 First, we state a few technical results that will be needed.   If $K: X \times  X \to \mathbb{C}$ is a positive definite kernel and $(\psi_{i})_{i \in \mathcal{I}}$ is a complete  orthonormal basis for $\mathcal{H}_{K}$, then it holds that
\begin{equation}\label{Mercergensum}
K(x,y)= \sum_{i \in \mathcal{I}}\psi_{i}(x)\overline{\psi_{i}(y)}, \quad x, y \in X.
\end{equation}

\begin{lem}\label{techequivmetricmercercompact} Let $X$ be a compact Hausdorff space such that there exists a  continuous conditionally negative definite metrizable  kernel. Then $X$ is  homeomorphic to a compact metric space and the RKHS of  any continuous positive definite kernel on $X$ is a separable space. 
\end{lem}

\begin{proof}Indeed, if $\gamma: X \times X \to \mathbb{R}$ is a continuous CND metrizable kernel, the metric 
$$
D_{\gamma}(x,y)=\sqrt{\gamma(x,y) - \gamma(x,x)/2 - \gamma(y,y)/2}
$$
is well defined. The inclusion $i : X \to X_{\gamma}$ is a continuous function because the kernel $D_{\gamma}$ is continuous. Conversely, since $X$ is compact and the topologies $X$ and $X_{\gamma}$ are Hausdorff, the inclusion must be a homeomorphism, and we can assume that $X$ is a compact metric space. In particular,  $X$ is a separable space.\\
The conclusion that the RKHS of any continuous positive definite kernel on $X$ must be separable is a consequence that $X$ is a separable space as proved in page $130$ in \cite{Steinwart}.\end{proof}

\begin{lem}\label{toptrick} Let $X$ be a  Hausdorff space and $\lambda$ be a nonzero measure in  $\mathcal{M}(X)$. Then there exists a sequence of nested compact sets $(\mathcal{C}_{n})_{n \in \mathbb{N}}$ for which $\lambda(A) = 0 $ for every Borel measurable set $A \subset X - \bigcup_{n \in \mathbb{N}}\mathcal{C}_{n}$ and $\overline{\bigcup_{n \in \mathbb{N}}\mathcal{C}_{n}}= Supp(\lambda)$.\\
In particular, if there exists a continuous conditionally negative definite metrizable kernel on $X$, then  the set $Supp(\lambda)$ is separable (induced topology)  and the RKHS of any continuous positive definite kernel on $X$ is separable when restricted to    $Supp(\lambda)$. 
\end{lem}

\begin{proof}On the first part we may assume that the measure $\lambda$ is nonnegative, because on the general case we can apply the result for the measures appearing on its Hahn decomposition.\\
Due to inner regularity,  there exists a sequence of nested compact sets $\mathcal{D}_{n}$, for which $
0 <  \lambda(\mathcal{D}_{n})$,  $\lim_{n \to \infty}\lambda(\mathcal{D}_{n})= \lambda(X) 
$.
Define  
$$
\mathcal{C}_{n} :=  \{ x \in\mathcal{D}_{n}, \quad \text{ every open set  that contains $x$ has positive measure} \}= \mathcal{D}_{n} \cap Supp(\lambda).
$$ 
Then $\mathcal{C}_{n}$ is compact, $\lambda(\mathcal{C}_{n})= \lambda(\mathcal{D}_{n})$ and $\lambda( \bigcup_{n \in \mathbb{N}}\mathcal{C}_{n}  )= \lambda(X)$. In particular, if $A \subset X - \bigcup_{n \in \mathbb{N}}\mathcal{C}_{n}$  is a measurable set then $\lambda(A)=0$. If $x \in Supp(\lambda)$, then every open set that contains $x$ has positive measure, in particular it must intersect $\bigcup_{n \in \mathbb{N}}\mathcal{C}_{n}$, because otherwise it would have zero measure, and then $x \in \overline{\bigcup_{n \in \mathbb{N}}\mathcal{C}_{n}}$. The fact that $ \overline{\bigcup_{n \in \mathbb{N}}\mathcal{C}_{n}} \subset Supp(\lambda)$ is a direct consequence of its definition.\\
If there exists a  continuous conditionally negative definite metrizable kernel on $X$, by Lemma \ref{techequivmetricmercercompact} each set $\mathcal{C}_{n}$ is separable (induced topology), but then the set $Supp(\lambda)$ is also separable. The conclusion that the RKHS of any continuous positive definite kernel on $X$ must be  separable when restricted to    $Supp(\lambda)$ is a consequence that $Supp(\lambda)$ is a separable space as proved in page $130$ in \cite{Steinwart}.\end{proof}

Lemma  \ref{techequivmetricmercercompact} and Lemma \ref{toptrick} are the main reason why we do not need to impose that the Hilbert space  $\mathcal{H}$ is separable.

The next two Lemmas are  used to simplify some arguments throughout the paper. A proof for the first one is simple and can be found at \cite{reescaling},  the second we present a brief explanation for it as it is a key technical result for the development of the subject.

\begin{lem}\label{simpliuniv}
	Let $X$ be a Hausdorff space, $K:X \times X \to \mathbb{R}$ be a continuous positive definite kernel and $f: X \to \mathbb{R}$ be a continuous function that is nowhere zero. The kernel 
	$$
	(x,y) \in   X \times X \to K_{f}(x,y)=f(x)K(x,y)f(y) \in \mathbb{R}
	$$ 
 is universal  if and only if the kernel $K$ is universal. Further, if the function $f$ is bounded and the kernel $K$ is ISPD then the kernel $K_{f}$ is ISPD.\end{lem}

\begin{lem} \label{prop17caponettogen}Let $X$ and $\tilde{X}$ be   Hausdorff spaces, $\tilde{K}: \tilde{X} \times \tilde{X} \to \mathbb{C}$ be an universal (or ISPD) positive definite kernel and $h: X \to \tilde{X}$  be a continuous function. The  positive definite kernel
	$$	
	(x,y) \in X \times X \to K(x,y):=\tilde{K}(h(x), h(y)) \in  \mathbb{C}
	$$	
	is universal (or ISPD) if and only if the function $h$ is injective.
\end{lem}
 
\begin{proof} If $x \neq y$ but $h(x)=h(y)$, the interpolation matrix of $K$ at $x,y$ is clearly non invertible, and then the kernel is not universal/ISPD.\\
For the converse we prove the ISPD case. Let $\lambda \in \mathcal{M}(X)$ and suppose that 
$$
0=\int_{X}\int_{X}\tilde{K}(h(x), h(y))d\lambda(x)d\overline{\lambda}(y).
$$
If $\lambda_{h}$ is the image measure of $\lambda$ through $h$, then $\lambda_{h} \in \mathcal{M}(\tilde{X})$, and
$$
0=\int_{\tilde{X}}\int_{\tilde{X}}\tilde{K}(z,w)d\lambda_{h}(x)d\overline{\lambda_{h}}(w).
$$
By the ISPD assumption on the kernel $\tilde{K}$, $\lambda_{h}$ is the zero measure, and this implies that $\lambda$ is the zero measure because for every compact set  $\mathcal{C} \subset X$ we have that $\lambda(\mathcal{C})= \lambda(h^{-1}[h(\mathcal{C})])=0$. \end{proof}

In order to prove the next result, we will use an infinite  dimensional version of the multinomial theorem. If $\mathcal{H}$ is a separable real Hilbert space and $(e_{i})_{i \in \mathbb{N}}$ is a complete orthonormal basis for it, then for every $n \in \mathbb{N}$
\begin{equation}\label{multinomial}
\langle x,y\rangle^{n}= \left (\sum_{i \in \mathbb{N}} x_{i}y_{i} \right )^{n} = \sum_{\alpha \in (\mathbb{N},\mathbb{Z}_{+}), |\alpha| =n} \frac{n!}{\alpha!} x^{\alpha}y^{\alpha}
\end{equation}
where $x_{i}= \langle x, e_{i} \rangle $, $(\mathbb{N},\mathbb{Z}_{+})$ is the space of functions from $\mathbb{N}$ to $\mathbb{Z}_{+}$, the condition $|\alpha|=n$ means that $\sum_{i \in \mathbb{N}}\alpha(i)=n$ (in particular $\alpha$ must be  the zero function except for a finite number of points). Also $\alpha! = \prod_{i \in \mathbb{N}}\alpha(i)!$ (which makes sense because $0!=1$) and $x^{\alpha}= \prod_{ \alpha(i) \neq 0}x_{i}^{\alpha(i)}$. This result can be proved using approximations of $\langle x,y\rangle^{n}$ on finite dimensional spaces and the multinomial theorem on those spaces.

\begin{proof}[\textbf{Proof of Theorem \ref{Gauszinho}}]
Note that it is sufficient to prove the case $\sigma = 1/2$ by Lemma \ref{prop17caponettogen}. Since  
$$	
G_{1/2}(x,y) = e^{-\langle x,x\rangle/2} e^{\langle x,y\rangle}e^{-\langle y,y\rangle/2 },	
$$	
Lemma \ref{simpliuniv} implies that the kernel $G_{1/2}$ is universal if and only if the kernel $e^{\langle x,y\rangle}$ is universal.\\
Let $\lambda \in \mathcal{M}_{c}(\mathcal{H})$ be a real valued measure, Lemma \ref{techequivmetricmercercompact} implies that the RKHS of the dot kernel  $\langle x,y\rangle$ is  separable  when restricted to  the compact set $X:=Supp(\lambda)$, that is, there exists a countable orthonormal set $(e_{i})_{i \in \mathbb{N}}$, for which
$$
\langle x,y\rangle= \sum_{i \in \mathbb{N}} \langle x, e_{i} \rangle \langle y, e_{i}\rangle := \sum_{i \in \mathbb{N}}x_{i}y_{i}, \quad x, y \in X.
$$
 If $0=\int_{X}\int_{X}e^{\langle x , y\rangle }d\lambda(x)d\lambda(y)$, then since the dot kernel is bounded on $X$, by the dominated convergence theorem we obtain that 
$$
0=\int_{X}\int_{X}e^{\langle x,y\rangle}d\lambda(x)d\lambda(y)=  \sum_{n \in \mathbb{Z}_{+}}\frac{1}{n!}\int_{X}\int_{X}\langle x,y\rangle^{n}d\lambda(x)d\lambda(y).
$$
Since the kernel 
$$
\langle x,y\rangle^{n}=\sum_{|\alpha|=n, \alpha \in (\mathbb{N},\mathbb{Z}_{+})} \frac{n!}{\alpha!} x^{\alpha}y^{\alpha}
$$
is positive definite, we have that 
\begin{align*}
0&=\int_{X}\int_{X}(\sum_{i \in \mathbb{N}} x_{i} y_{i})^{n}d\lambda(x)d\lambda(y)= \sum_{|\alpha|=n, \alpha \in (\mathbb{N},\mathbb{Z}_{+})}\frac{n!}{\alpha!} \int_{X}\int_{X} x^{\alpha} y^{\alpha}d\lambda(x)d\lambda(y)\\
&=\sum_{|\alpha|=n, \alpha \in (\mathbb{N},\mathbb{Z}_{+})}\frac{n!}{\alpha!} \left |\int_{X}x^{\alpha} d\lambda(x)\right |^{2},
\end{align*}
then $\int_{X}x^{\alpha} d\lambda(x) =0$ for every $\alpha \in (\mathbb{N},\mathbb{Z}_{+})$. The algebra of continuous functions 
$$
\mathcal{A}:=\{ x \in  X \to x^{\alpha } \in \mathbb{R}, \quad \alpha \in (\mathbb{N},\mathbb{Z}_{+})  \}
$$
separates points, because  if $x,y$ are not separated by the algebra $\mathcal{A}$, then
$$
2\langle x,y \rangle = 2\sum_{i \in \mathbb{N}} x_{i}y_{i}=\sum_{i \in \mathbb{N}} x_{i}x_{i} + y_{i}y_{i} = \langle x,x \rangle  + \langle y,y \rangle 
$$
 which can only occur if $x=y$. Since the constant function equal to $1$ belongs to the algebra $\mathcal{A}$, the Stone-Weierstrass Theorem implies that $span\{ x \in  X \to x^{\alpha } \in \mathbb{R} \}$ is dense on $C(X)$, and consequently $\lambda$ must be the zero measure, implying that the kernel $G_{1/2}$ is universal.  
 \end{proof}

\begin{proof}[\textbf{ Proof of Theorem \ref{Gauszao}}]
The proof follows by similar arguments (also notation) as the one we used at the proof of  Theorem \ref{Gauszinho} and several applications of the  dominated convergence theorem. Again, it is sufficient to prove the case $\sigma = 1/2$ by Lemma \ref{prop17caponettogen}. Let   $\lambda \in \mathcal{M}(\mathcal{H})$ be a real valued measure    such that
\begin{equation}\label{Gauszaoa}
0=\int_{\mathcal{H}}\int_{\mathcal{H}}e^{\langle x,y \rangle}e^{-\langle x,x \rangle/2}e^{-\langle y,y \rangle/2}d\lambda(x)d\lambda(y).
\end{equation}
  Lemma \ref{toptrick} implies that the RKHS of the dot kernel  $\langle x,y\rangle$ is  separable  when restricted to  the  set $X:=\overline{Span (Supp(\lambda))}$, then  there exists a countable orthonormal set $(e_{i})_{i \in \mathbb{N}}$ in $\mathcal{H}$, for which
$$
\langle x,y\rangle= \sum_{i \in \mathbb{N}} \langle x, e_{i} \rangle \langle y, e_{i}\rangle := \sum_{i \in \mathbb{N}}x_{i}y_{i}, \quad x, y \in X.
$$ Note that  $|e^{\langle x,y \rangle}e^{-\langle x,x \rangle/2}e^{-\langle y,y \rangle/2}| \leq 1$ and 
\begin{align*}
&\sum_{n =0}^{m}\left |\frac{1}{n!}\langle x,y \rangle^{n}e^{-\langle x,x \rangle/2}e^{-\langle y,y \rangle/2}\right |\\
& \quad \leq e^{-\langle x,x \rangle/2}e^{-\langle y,y \rangle/2} \left (\sum_{n =0}^{m}\frac{1}{n!}\left (\frac{\langle x,x \rangle}{2} + \frac{\langle y,y \rangle}{2}  \right )^{n} \right ) \leq 1,
\end{align*}
 so Equation \ref{Gauszaoa} is equivalent at
\begin{equation}\label{Gauszaob}
0=\int_{X}\int_{X}\langle x,y \rangle^{n}e^{-\langle x,x \rangle/2}e^{-\langle y,y \rangle/2}d\lambda(x)d\lambda(y), \quad n \in \mathbb{Z}_{+}.
\end{equation}
Since 
$$
\langle x,y \rangle^{n} = \sum_{|\alpha|=n, \alpha \in (\mathbb{N},\mathbb{Z}_{+})} \frac{n!}{\alpha!} x^{\alpha}y^{\alpha}, \quad x, y \in X,
$$
Equation \ref{Gauszaob} is equivalent at
\begin{equation}\label{Gauszaoc}
0=\int_{X}x^{\alpha}e^{-\langle x,x \rangle/2}d\lambda(x), \quad \alpha \in (\mathbb{N},\mathbb{Z}_{+}), |\alpha| < \infty
\end{equation}
because $|\langle x,y \rangle^{n} e^{-\langle x,x \rangle/2}e^{-\langle y,y \rangle/2}| \leq n!$, also
\begin{align*}
2&\sum_{|\alpha|=n, \alpha \in (\mathbb{N},\mathbb{Z}_{+})}\frac{n!}{\alpha!}\left |x^{\alpha}y^{\alpha}e^{-\langle x,x \rangle/2}e^{-\langle y,y \rangle/2}\right |\\
&\leq e^{-\langle y,y \rangle/2}e^{-\langle x,x \rangle/2}\sum_{|\alpha|=n, \alpha \in (\mathbb{N},\mathbb{Z}_{+})}\frac{n!}{\alpha!} (x^{2\alpha} +  y^{2\alpha}) \\
& \leq e^{-\langle y,y \rangle/2}e^{-\langle x,x \rangle/2}(\langle x,x \rangle^{n} + \langle y,y \rangle^{n})  \leq n!,
\end{align*}
and consequently
$$
\int_{X}\int_{X}\langle x,y \rangle^{n}e^{-\langle x,x \rangle/2}e^{-\langle y,y \rangle/2}d\lambda(x)d\lambda(y)= \sum_{|\alpha|=n, \alpha \in (\mathbb{N},\mathbb{Z}_{+})} \frac{n!}{\alpha!} \left |   \int_{X}x^{\alpha}e^{-\langle x,x \rangle/2}d\lambda(x)\right|^{2}.
$$
$\circ$ (\textbf{Affirmation $1$}) We claim that 
\begin{equation}\label{gausssunboundc0eq1}
\int_{X} x^{\alpha}e^{-r\langle x , x \rangle /2}d\lambda(x)=0
\end{equation}
for every $\alpha \in (\mathbb{N},\mathbb{Z}_{+})$ with $|\alpha| < \infty$ and $r>0$. Indeed, Equation \ref{Gauszaoc} implies that Equation \ref{gausssunboundc0eq1} is valid for $r=1/2$, and we use an induction type of  argument to prove  the general case. Suppose that Equation \ref{gausssunboundc0eq1} holds for a $r^{\prime}>0$, we claim that it also holds for every $r \in (0, 2r^{\prime})$. Indeed,    for every  $\beta \in (\mathbb{N},\mathbb{Z}_{+})$ with $|\beta| < \infty$
$$
\int_{X}x^{\alpha}x^{2\beta} e^{-r^{\prime}\langle x , x \rangle} d\lambda(x) =0.
$$
By the  dominated convergence theorem 
\begin{align*}
0&=\sum_{|\beta|=n, \beta \in (\mathbb{N},\mathbb{Z}_{+})} \frac{n!}{\beta!}\int_{X}x^{\alpha}  x^{\beta}x^{\beta} e^{-r^{\prime}\langle x , x \rangle} d\lambda(x) = \int_{X}x^{\alpha} \left ( \sum_{|\beta|=n, \beta \in (\mathbb{N},\mathbb{Z}_{+})} \frac{n!}{\beta!}x^{\beta}x^{\beta}\right )e^{-r^{\prime}\langle x , x \rangle} d\lambda(x)\\
&= \int_{X}x^{\alpha} \langle x , x \rangle^{n}e^{-r^{\prime}\langle x , x \rangle} d\lambda(x) .
\end{align*}
 In particular, by applying once again the dominated convergence theorem, we obtain that for every $|s|<r^{\prime}$ 
\begin{align*}
0&= \sum_{n \in \mathbb{Z}_{+}}\frac{s^{n}}{n!}\int_{X}x^{\alpha} \langle x , x \rangle^{n}e^{-r^{\prime}\langle x , x \rangle} d\lambda(x) = \int_{X}x^{\alpha} \left ( \sum_{n \in \mathbb{Z}_{+}}\frac{s^{n}}{n!} \langle x , x \rangle^{n}\right )e^{-r^{\prime}\langle x , x \rangle} d\lambda(x)\\ 
&=  \int_{X}x^{\alpha} e^{( s-r^{\prime})\langle x , x \rangle} d\lambda(x),
\end{align*}
and so our claim is true on the interval $r \in (0, 2r^{\prime})$ by choosing $s = r- r^{\prime}$.\\
$\circ$ (\textbf{Affirmation $2$}) We claim that
\begin{equation}\label{Gauszaod}
0=\int_{X}\prod_{\mu =1}^{m}\left (e^{\langle x,z_{\mu} \rangle }e^{-\langle x,x \rangle/2} \right )d\lambda(x), 
\end{equation}
for whichever $m \in \mathbb{N}$ and $z_{1}, \ldots , z_{m} \in X$ (not necessarily distinct). Indeed, if $z = z_{1} + \ldots +z_{m}$
$$
\sum_{n=0}^{\infty}\frac{1}{n!}\int_{X} \langle x, z \rangle^{n}e^{-m\langle x,x \rangle/2}  d\lambda(x) =  \int_{X} e^{\langle x,z\rangle }e^{-m\langle x,x \rangle/2}d\lambda(x)=\int_{X}\prod_{\mu =1}^{m}\left (e^{\langle x,z_{\mu}\rangle }e^{-\langle x,x \rangle/2}\right )d\lambda(x),
$$
 because  $|e^{\langle x,z\rangle }e^{-m\langle x,x \rangle/2}|\leq e^{\langle z,z\rangle /2}e^{-(m-1)\langle x ,x\rangle /2} $  and
\begin{align*}
&\sum_{n=0 }^{\infty}\frac{1}{n!}\left |\langle x, z \rangle^{n}e^{-m\langle x,x \rangle/2} \right |\\
& \quad \leq  \sum_{n=0 }^{\infty}\frac{1}{n!} \left (\frac{\langle x,x \rangle + \langle z, z\rangle }{2} \right )^{n}e^{-m\langle x,x \rangle/2} \leq e^{\langle z,z\rangle /2}e^{-(m-1)\langle x ,x\rangle /2}.
\end{align*}
Similarly, 
$$
\sum_{|\alpha|=n, \alpha \in (\mathbb{N},\mathbb{Z}_{+})} \frac{n!}{\alpha!}z^{\alpha}  \int_{X} x^{\alpha}e^{-m\langle x,x \rangle /2} d\lambda(x)=\int_{X} \langle x, z \rangle^{n}e^{-m\langle x,x \rangle/2}  d\lambda(x), 
$$
because $|\langle x, z \rangle^{n}e^{-m\langle x,x \rangle/2}| \leq 2^{-n}(\langle x, x \rangle + \langle z, z \rangle)^{n}e^{-m\langle x,x \rangle/2}$ and 
\begin{align*}
 2&\sum_{|\alpha|=n, \alpha \in (\mathbb{N},\mathbb{Z}_{+})}  \frac{n!}{\alpha!}\left | x^{\alpha} z^{\alpha}e^{-m\langle x,x \rangle/2} \right |\\
 &\leq   \sum_{|\alpha|=n, \alpha \in (\mathbb{N},\mathbb{Z}_{+})}  \frac{n!}{\alpha!} (x^{2\alpha}+  z^{2\alpha})e^{-m\langle x,x \rangle/2} = \langle x,x \rangle^{n} +  \langle z,z \rangle^{n})e^{-m\langle x,x \rangle}. \end{align*}
The conclusion follows from Affirmation $(1)$.
Now, consider the algebra of functions generated by the set
$$
\mathcal{A} :=span\{x \in X  \to e^{\langle x,z \rangle}e^{-\langle x,x \rangle/2} \in \mathbb{R},\quad  z \in X \}.
$$
 Equation \ref{Gauszaod} implies that for every $h \in \mathcal{A}$  it holds that $\int_{X}h(x)d\lambda(x)=0$. Moreover:
\begin{enumerate}
    \item[(i)]There exists $h \in \mathcal{A}$ for which $h(x)>0$ for every $x \in X$.
    \item[(ii)] If $z_{1}, z_{2} \in X$,  $B(z_{1},R_{1})$ and $B(z_{2}, R_{2})$ are two disjoint open balls of $X$, there exists $h \in \mathcal{A}$ for which $h(x)>0$ on $B(z_{1},R_{1}) \cap X$ and $h(x)<0$ on $B(z_{2},R_{2})\cap X$.
\end{enumerate}
For $(i)$, take $h$ as any of the functions $x \to e^{\langle x,z \rangle}e^{-\langle x,x \rangle/2}$. As for $(ii)$, define
\begin{align*}
h(x) &= e^{-\langle z_{1},z_{1} \rangle/2}( e^{\langle x,z_{1} \rangle}e^{-\langle x,x \rangle/2} )e^{R_{1}^{2}/2} -e^{-\langle z_{2},z_{2} \rangle/2}( e^{\langle x,z_{2} \rangle}e^{-\langle x,x\rangle/2} )e^{R_{2}^{2}/2}\\
&=e^{-\| x - z_{1} \|^{2}/2 + R_{1}^{2}/2} -e^{-\| x - z_{2} \|^{2}/2 + R_{2}^{2}/2}. 
\end{align*}
Theorem \ref{Farrelproof} implies that $\mathcal{A}$ is dense on $L^{1}(X,|\lambda|)$. If $P^{+}, N^{-}$ is a Hahn decomposition for the measure $\lambda$, the continuous linear functional 
$$
h \in L^{1}(X,|\lambda|) \to \int_{X}h(x)(\chi_{P^{+}}(x)- \chi_{N^{-}}(x))d|\lambda|(x)= \int_{X}h(x)d\lambda(x) \in \mathbb{R}
$$
is zero on $\mathcal{A}$, which can only happen if $\chi_{P^{+}} = \chi_{N^{-}}$ on $L^{\infty}(X, |\lambda|)$, but then $\lambda$ must be the zero measure, which concludes the proof. 
\end{proof}

\begin{proof}[\textbf{Proof of the Lemma \ref{Gausskerinfty} }]We only prove the case $\sigma=1$, the others follows by a simple change of notation.\\
Suppose that $(i)$ holds, in particular, the function $e^{-\|x-z_{0}\|^{2}/2}$ also belongs to $C_{0}(X)$. By the parallelogram law
$$
e^{-\|x-z\|^{2}}= e^{-\|x-z_{0}\|^{2}/2 -  \|x+z_{0} - 2z\|^{2}/2 + \|z-z_{0}\|^{2}} \leq e^{\|z-z_{0}\|^{2}}e^{-\|x-z_{0}\|^{2}/2},
$$
implying that the function $e^{-\|x-z\|^{2}}$ belongs to $C_{0}(X)$ . It is immediate that $(ii)$ implies $(i)$\\ 
Relations $(ii)$ and $(iii)$ are equivalent by Proposition \ref{rkhscontained}.\\
If $(iv)$ holds, then for every $z \in X$ and $\epsilon >0$, the set $\{x \in X, \quad e^{-\|x-z\|^{2}}\geq \epsilon \}$ is bounded and closed on $X$, so it must be compact by the hypothesis implying that the function $G_{\sigma, z} \in C_{0}(X)$ for every $z \in X$. For the converse, it is sufficient to show  that if $z \in X$ and $r>0$ then  the closed ball $\|x-z\|\leq r$ is compact, but this is  the set of points  that satisfies  $e^{-\|x-z\|^{2}} \geq e^{-r^{2}}$, which is compact by the $C_{0}$ hypothesis.
\end{proof}

 \begin{proof}[\textbf{Proof of Theorem \ref{gausssunboundc0}}] It is sufficient to prove the case $\sigma=1/2$ by Lemma \ref{prop17caponettogen}. If the kernel is $C_{0}$-universal, by definition  its necessary that $\mathcal{H}_{G_{1/2}} \subset C_{0}(X)$.     Conversely, if $\mathcal{H}_{G_{1/2}} \subset C_{0}(X)$  we only need to to prove that the kernel is ISPD  (on the sigma algebra $\mathscr{B}(X)$). We present a proof that does not involve Theorem \ref{Farrelproof},  instead we use the $C_{0}$ version of the Stone-Weierstrass Theorem, which can be found at Section $4.7$ at \cite{folland}.\\ 
The arguments are the same as   the one of Theorem \ref{integrallygaussianonmetric} up to Equation \ref{gausssunboundc0eq1} (we do not use Affirmation $2$). 
The algebra of continuous functions on $X$
$$
\mathcal{A}:= \{x^{\alpha } e^{-r\langle x , x \rangle},\quad   r \in (0, \infty), \quad \alpha \in (\mathbb{N},\mathbb{Z}_{+}), |\alpha|< \infty    \} \subset C_{0}(X).
$$
The function $h(x)= e^{-\langle x , x \rangle} \in \mathcal{A}$ is such that $h(x)>0$ for every $x \in X$, also, the algebra $\mathcal{A}$ separates points because if $x^{\alpha}e^{-\langle x , x \rangle}= y^{\alpha}e^{-\langle y , y \rangle}$ for every $\alpha \in (\mathbb{N},\mathbb{Z}_{+}), |\alpha|< \infty$, then we must have that $\langle x , x \rangle= \langle x , y \rangle= \langle y , y \rangle$, which  can only occur if $x=y$. As a direct consequence of the $C_{0}$ version of the Stone-Weierstrass Theorem we obtain that $\lambda$ must be the zero measure, proving that the kernel $G_{1/2}$ is $C_{0}(X)$-universal.
\end{proof}

\subsection{Section \ref{Universality of Schoenberg-Gaussian kernels}}

\begin{proof}[\textbf{Proof of Theorem \ref{improvSchoen}}]
Since $\gamma$ is a conditionally negative definite kernel, by Equation \ref{condequa} there exists a Hilbert space $\mathcal{H}$ and functions $h: X  \to \mathcal{H}$ and  $f : X \to \mathbb{R}$ for which
$$
\gamma(x,y)=  f(x) + \|h(x) - h(y)\|^{2} + f(y).
$$
Note that that the functions $f,h$ are continuous because $f(x)= \gamma(x,x)/2$, and $\|h(x) -h(y) \|= \sqrt{\gamma(x,y) - f(x) - f(y)}$. \\
If the kernel $G_{\gamma}$ is universal, then the kernel $G_{\gamma}$ is SPD by definition.\\
If the kernel $G_{\gamma}$ is SPD the matrix
$$
\begin{bmatrix}
e^{-\gamma (z,z)} & e^{-\gamma (z,w)} \\ 
e^{-\gamma (w,z)} & e^{-\gamma (w,w)}  
\end{bmatrix}
$$
is  positive definite and its determinant is equal to $e^{-\gamma(z,z) - \gamma(w,w)} ( 1 - e^{\gamma(z,z) + \gamma(w,w) -2\gamma(z,w) } )$, which is nonzero for every $z \neq w$ if and only if $\gamma$ is metrizable.  

It only remains to prove that if $\gamma$ is metrizable  the kernel $G_{\gamma}$ is universal. Since
$$	
G_{\gamma}(x,y) = e^{-f(x)} e^{-\|h(x) - h(y)\|^{2}}e^{ -f(y) },	
$$	
Lemma \ref{simpliuniv} implies that the kernel $G_{\gamma}$ is universal if and only if the kernel $e^{-\|h(x) - h(y)\|^{2}}$ is universal. The metrizability is equivalent to the injectivity of the function $h$, being so,  the kernel is universal by Proposition \ref{prop17caponettogen} and Theorem \ref{Gauszinho}.
\end{proof}

\begin{proof}[\textbf{ Proof of Theorem \ref{integrallygaussianonmetric}}] If $G_{\gamma}$ is ISPD the function $x \in X \to \gamma(x,x)$ must be bounded from below in order to $G_{\gamma}$ be a bounded kernel. Since  every ISPD kernel is an SPD kernel,   Theorem \ref{improvSchoen}  implies that   $\gamma$ is metrizable.\\
Conversely, since $\gamma$ is a continuous conditionally negative definite kernel, by Equation \ref{condequa} there exists a Hilbert space $\mathcal{H}$ and continuous functions $h: X  \to \mathcal{H}$ and  $f : X \to \mathbb{R}$ for which
$$
\gamma(x,y)=  f(x) + \|h(x) - h(y)\|^{2} + f(y).
$$
Note that the kernel $e^{-\|h(x) - h(y)\|^{2}}$ is ISPD, because $h$ is continuous  and injective, Lemma \ref{simpliuniv} and Theorem \ref{Gauszao}. Also, the function $x \to e^{ -f(x) }$ is bounded because the function $x \in X \to \gamma(x,x)$ is bounded from below.      Proposition \ref{prop17caponettogen}  implies that  
$$	
G_{\gamma}(x,y) = e^{-f(x)} e^{-\|h(x) - h(y)\|^{2}}e^{ -f(y) },	
$$	
is an ISPD kernel.
\end{proof}

\begin{proof}[\textbf{Proof of the Lemma \ref{condnegkerinfty}}]  Since the function $\gamma(x,x)$ is bounded and using Equation \ref{condequa}, for every $z \in X$ the function $ x \to e^{-\gamma(x,z)} \in C_{0}(X)$ if and only if $x \to e^{-\|h(x) - h(z)\|^{2}} \in C_{0}(X)$.   By the parallelogram law, we have that  
$$
e^{-\|h(x) - h(z)\|^{2}}= e^{-\|h(x)-h(z_{0})\|^{2}/2 -  \|h(x)+h(z_{0}) - 2h(z)\|^{2}/2 + \|h(z)-h(z_{0})\|^{2}} \leq e^{-\|h(x)-h(z_{0})\|^{2}/2}e^{\|h(z)-h(z_{0})\|^{2}}  ,
$$
implying that the function the function $ x \to e^{-\|h(x) - h(z)\|^{2}} \in C_{0}(X)$ for some $z \in X$ if and only if $x \to e^{-\|h(x) - h(z)\|^{2}} \in C_{0}(X)$ for every $z \in X$. Proposition \ref{rkhscontained} concludes the proof.\end{proof}

\begin{proof}[\textbf{Proof of Theorem \ref{gaussunboundc0}}] If the kernel is $C_{0}(X)$-universal then it is SPD, which by Theorem \ref{improvSchoen} the kernel $\gamma$ must be metrizable. Also,  $\mathcal{H}_{G_{\gamma}}\subset C_{0}(X)$ by definition.\\
Conversely, if  $\mathcal{H}_{G_{\gamma}}\subset C_{0}(X)$,  $G_{\gamma}$ is $C_{0}$-universal if and only if it is  ISPD, which occur due to Theorem Theorem \ref{integrallygaussianonmetric}.\end{proof}

\begin{proof}[\textbf{Proof of Theorem \ref{gaussispd}}]Since 
$$
f(\gamma(x,y))= \int_{[0, \infty)}e^{-r\gamma(x,y)}d\lambda(r), 
$$
where $\lambda$ is a nonnegative measure in   $\mathcal{M}([0, \infty))$,  if $\gamma$ is metrizable then $e^{-\gamma(x,y)}$ is an ISPD kernel and if $f$ is not the constant function (which is equivalent at $\lambda (0, \infty)>0$), then  $f(\gamma(x,y))$ is ISPD by  Theorem $3.7$ in \cite{jean2020}.
\end{proof}

\subsection{Products of positive definite kernels}\label{Products of positive definite kernels}

The Schur product Theorem asserts that if $p, q: X \times X \to \mathbb{C}$ are positive definite kernels, then their product kernel
$$(x,y) \in X \times X \to (p\odot q)(x,y) :=p(x,y)q(x,y)$$
is positive definite. This result is a direct consequence that the Hadamard Product of positive definite kernels is positive definite, where if $p: X \times X \to \mathbb{C}$, $q: Z \times Z \to \mathbb{C}$ are positive definite kernels, its Hadamard product is the kernel
$$
((x,z),(y,w)) \in (X \times Z) \times  (X \times Z) \to (p\otimes q)((x,z),(y,w)) :=p(x,y)q(z,w)
$$

In this section we prove some results concerning the relation between the Schur/Hadamard product of kernels and the concepts of SPD/universality/ISPD. We also  present a weighted version of these results. 

We emphasize that in this section we implicitly assume that the domain of the kernels is  a Hausdorff space.

\begin{lem}\label{Schurmulgen1}Let  $p, q: X \times X \to \mathbb{C}$ be  continuous positive definite kernels. Suppose that  the kernel $p$ is SPD/universal/ISPD, then a measure $\lambda \in \mathcal{M}_{\delta}(X)/\mathcal{M}_{c}(X)/\mathcal{M}(X)$ satisfies 
$$
\int_{X}\int_{X}p(x,y)q(x,y)d \lambda(x)\overline{\lambda}(y)=0. 
$$	
if and only if  
$$
\int_{A}\int_{A}q(x,y)d \lambda(x)\overline{\lambda}(y)=0,  
$$
for every  $A \in \mathscr{B}(X)$.  In particular,  the kernel $p\odot q$ is SPD/universal/ISPD if and only if $q(x,x)>0$ for every $x \in X$.
\end{lem}


\begin{proof}  The proof for the three cases are identical, so we only focus on the $\mathcal{M}(X)$ case. Let $\lambda \in  \mathcal{M}(X)$ be such that
\begin{equation}\label{Schurmulgen1eq1}
\int_{X}\int_{X}p(x,y)q(x,y)d \lambda(x)\overline{\lambda}(y)=0. 
\end{equation}	
The continuous conditionally negative definite kernel 
$$
(x,y) \in X \times X \to p(x,x) -p(x,y) - p(y,x)+ p(y,y) \in \mathbb{R} 
$$
is metrizable. Because of that, Lemma \ref{toptrick} implies that the kernel $q$ can be written as $q(x,y)= \sum_{k \in \mathbb{N}}q_{k}(x)\overline{q_{k}(y)}$ for $x,y \in Z:=Supp(\lambda)$, and then Equation \ref{Schurmulgen1eq1}  is equivalent to 
$$
\int_{Z}\int_{Z}p(x,y)q_{k}(x) \overline{q_{k}(y)}d \lambda(x)\overline{\lambda}(y)=0, \quad k \in \mathbb{N}
$$
but the kernel $p$ is ISPD, so the previous relation is equivalent to the measures $ q_{k}d\lambda$ (note that $ q_{k}d\lambda$ belongs to the same  space of measures as $\lambda$) being zero for every $k \in \mathbb{N}$. Using once again the series representation for $q$, all the measures  $ q_{k}d\lambda$ are zero if and only if 
$$
\sum_{k \in \mathbb{N}}\left |\int_{A}q_{k}(x)d\lambda(x)\right |^{2}=  \int_{A}\int_{A}q(x,y)d \lambda(x)\overline{\lambda}(y)=0,  \quad A \in \mathscr{B}(X),
$$
which proves our claim.\\
Now, suppose in addition that $q(x,x)>0$ for every $x \in X$. In this case by the continuity of the function $q$,  for every $z \in X$  there exists an open  set $U_{z}$ that contains $z$ for which $q(x,y)>0$ for every $x,y \in U_{z}$. If $X^{+, Re}$, $X^{-, Re}$, $X^{+,Im}$, $X^{-, Im}$ is a Hahn decomposition of the set $X$ by the measure $\lambda$, then $ U_{z}\cap X^{+, Re} \in \mathscr{B}(X)$ and
$$
\int_{U_{z}\cap X^{+, Re}}\int_{U_{z}\cap X^{+, Re}}q(x,y)d \lambda(x)\overline{\lambda}(y)=0.
$$
But the integrand is a positive function and the measure $\lambda$ is nonnegative on $X^{+, Re}\cap U_{z}$, which implies that this double integral is zero if and only if the measure $\lambda$ is the zero measure on the set $X^{+, Re}\cap U_{z}$. Suppose by an absurd that the measure $\lambda^{+, Re}$ is nonzero and let  $\mathcal{C}$ be an arbitrary compact  set on $X$. Then there exists a finite set $z_{1}, \ldots , z_{n} \in X$ for which $\mathcal{C} \subset \bigcup_{k =1}^{n}U_{z_{k}}$. Note that
$$
\lambda^{+, Re}(\mathcal{C})= \lambda^{+, Re}(\mathcal{C}\cap X^{+, Re})\leq \lambda^{+, Re}((\bigcup_{k =1}^{n}U_{z_{k}})\cap X^{+, Re})\leq \sum_{k =1}^{n}\lambda^{+, Re}(U_{z_{k}}\cap X^{+, Re})=0,  
$$
which is an absurd by the inner regularity of $\lambda^{+, Re}$. Conversely, if $p\odot q$ is ISPD, then we must have that $[p\odot q](x,x)>0$ for every $x \in X$, and then  $q(x,x)>0$ for every $x \in X$.
\end{proof}

Now we focus on the Hadamard product. Before that we prove a measure theoretical Lemma that will simplify the arguments.

\begin{lem}\label{measuprodprojection} Let $X, Z$ be Hausdorff spaces and $\lambda \in \mathcal{M}_{\delta}(X\times Z)$, $\mathcal{M}_{c}(X\times Z)$, $\mathcal{M}(X\times Z)$. If  the function $\phi: X \times Z \to \mathbb{C}$ is  bounded and continuous  then  the function
$$
A \in \mathscr{B}(X) \to \lambda_{\phi}(A):= \int_{X \times Z} \chi_{A}(x)\phi(x,y)d\lambda(x,y) \in \mathbb{C} 
$$  
is a finite measure on  $\mathcal{M}_{\delta}(X)$, $\mathcal{M}_{c}(X)$, $\mathcal{M}(X)$. \end{lem}
\begin{proof}We focus the arguments on the $\mathcal{M}(X \times Z)$ case, being the others similar.\\
The fact that $\lambda_{\phi}$ is a measure is obtained by a direct application  of the  dominated convergence theorem.\\
Note that $\lambda_{\phi}$ is the linear combination of $16$ measures (for instance, $ \chi_{A}\phi^{+, Re}d\lambda^{+, Re}$), so we can assume that $\lambda$ is a nonnegative measure and $\phi$ is a nonnegative function. In particular,  if $A \subset B$ then
\begin{align*}
|\lambda_{\phi}(B) - \lambda_{\phi}(A)|&=  \int_{X \times Z}(\chi_{B}(x) - \chi_{A}(x))\phi(x,y)d\lambda(x,y)\\
&\leq sup_{(x,y) \in X \times Z}|\phi(x,y)|(\lambda(B\times Z) - \lambda(A\times Z))
\end{align*}
and we prove the outer and inner regularity of $\phi_{\lambda}$ by showing that the finite measure $A \in \mathscr{B}(X) \to \lambda(A\times Z) $ is inner and outer regular.\\
$\circ$(Inner regular) Let $\epsilon >0$ and  $E \in \mathscr{B}(X)$. By the inner regularity of $\lambda$ on the set $E \times Z$ there exists a compact set   $\mathcal{C} \subset E \times Z$  for which $\lambda(E \times Z) - \lambda(\mathcal{C})< \epsilon$. The compact set $C:=\pi_{1}(\mathcal{C}) \subset X$ (projection on the first variable) is such that $ \mathcal{C}\subset  C \times Z \subset E \times Z $, and then $ \lambda(E \times Z) - \lambda(C \times Z) < \epsilon$.\\
$\circ$ (Outer regular) Let $\epsilon >0$, $E \in \mathscr{B}(X)$. By the inner regularity of $\lambda$ there exists  compact sets $\mathcal{C}_{1}$ of $X$ and $\mathcal{C}_{2}$ of $Z$ for which  $\lambda(X \times Z ) - \lambda(\mathcal{C}_{1}\times \mathcal{C}_{2}) < \epsilon$.  By the outer regularity of $\lambda$ there exists an  open set $U$ that contains $E \times \mathcal{C}_{2}$  and $\lambda(U) - \lambda(E \times \mathcal{C})< \epsilon$. For every $x \in E$ there exists an open set $U_{x}$  on $X$ that contains $x$ for which $U_{x}\times \mathcal{C}_{2} \subset U$, because the set $\{x\}\times \mathcal{C}_{2}$ is compact. Define the open set $V = \bigcup_{x \in E}U_{x}$ of $X$ and note that $E \times \mathcal{C}_{2} \subset V \times\mathcal{C}_{2} \subset U$ and
\begin{align*}
0\leq \lambda(V\times Z) - \lambda(E\times Z) &< 2\epsilon +   \lambda((V\times Z) \cap (\mathcal{C}_{1}\times \mathcal{C}_{2}) )  - \lambda((E\times Z)\cap (\mathcal{C}_{1}\times \mathcal{C}_{2}))\\
&<   2\epsilon +   \lambda((V\cap \mathcal{C}_{1})\times \mathcal{C}_{2})  - \lambda((E\cap \mathcal{C}_{1})\times \mathcal{C}_{2})\\
& < 4\epsilon + \lambda(V\times \mathcal{C}_{2})  - \lambda(E\times \mathcal{C}_{2}) < 5\epsilon. 
\end{align*}
\end{proof}

On the next Lemma we use an equivalent condition for a positive definite kernel $K: X \times X \to \mathbb{C}$ to be SPD/Universal/ISPD, which occurs if and only if the only measure \\ $ \lambda \in \mathcal{M}_{\delta}(X )/ \mathcal{M}_{c}(X )/\mathcal{M}(X )$ for which
$$
\int_{X}K(x,y)d\lambda(x)=0, \quad y \in X
$$
is the zero measure  \cite{micchelli2006universal}, \cite{Sriperumbudur3}.
\begin{lem}\label{Kroengen1}Let $p : X \times X\to \mathbb{C}$ and $q:Z \times Z \to \mathbb{C} $ be bounded  positive definite  continuous kernels. Suppose that the  kernel $p$ is SPD/universal/ISPD, then a  measure  $\lambda \in \mathcal{M}_{\delta}(X \times Z)/ \mathcal{M}_{c}(X \times Z)/ \mathcal{M}(X \times Z)$ (respectively) satisfies 
$$
\int_{X\times Z}\int_{X\times Z}p(x,y)q(u,v)d\lambda(x,u)d\overline{\lambda}(y,v)=0 
$$
if and only if 
	$$
	\int_{Z}\int_{Z}q(u,v)d \lambda_{A}(u)d\overline{\lambda_{A}}(v)=0, 
	$$
	 for every  $A \in \mathcal{B}(X)$ where $\lambda_{A}$ is the measure in $ \mathcal{M}_{\delta}(Z), \mathcal{M}_{c}( Z), \mathcal{M}( Z)$ (respectively) for which $B \in \mathscr{B}(Z) \to \lambda_{A}(B):=\lambda(A\times B)$. In particular, the kernel $p\otimes q$ is SPD/universal/ISPD if and only if the only  the same occur with the kernels $p,q$. 
\end{lem}


\begin{proof}
The proof for the three cases are identical, so we only focus on the ISPD case. Let $\lambda \in   \mathcal{M}(X \times Z)$ be such that
	$$
	\int_{X \times Z}p(x,y)q(u,v)d \lambda(x,u)=0, 
	$$
for every $(y,v) \in X\times Z$. Since $p$ is an ISPD kernel and the measure
$$
A \in \mathscr{B}(X) \to  \int_{X \times Z} \chi_{A}(x)q(u,v)d \lambda(x,u) \in \mathbb{C} 
$$
is an element  of $\mathcal{M}(X)$ for every $v \in Z$ by Lemma \ref{measuprodprojection}, then this measure is the zero measure for every $v \in Z$. Note that the measure $\lambda_{A}$ is an element of  $\mathcal{M}(Z)$ (this is obtained from Lemma \ref{measuprodprojection} by reversing the roles of $X$ and $Z$ and taking $\phi$ as the constant one function) and 
$$
\int_{Z}q(u,v)d\lambda_{A}(u)= \int_{X\times Z}\chi_{A}(x)q(u,v)d\lambda(x,u)=0, \quad v \in Z. 
$$
In particular, if $q$ is an ISPD kernel, the measure $\lambda_{A}$ must be zero for every $A \in \mathscr{B}(X)$, which implies that $\lambda$ is the zero measure.\\
Conversely, if $p\otimes q$ is an ISPD kernel $p$ and $q$ must also be integrally positive definite  because 
$$
\{ \lambda_{1}\times \lambda_{2}, \quad \lambda_{1} \in  \mathcal{M}(X), \lambda_{2} \in \mathcal{M}(Z)\} \subset \mathcal{M}(X\times Z) 
$$
and Fubini-Tonelli Theorem.\end{proof}

Lastly, we prove a result that elucidates when a weighted sum of positive definite kernels is SPD/universal/ISPD. 

\begin{lem} \label{Hadamardintegral} Let $\Omega$ be a Hausdorff space, $\eta$ be a nonnegative $\sigma$-finite Radon measure on it and a family of bounded positive definite kernels $(p_{w})_{w \in \Omega}$ on $X$ such that $p: \Omega \times (X \times X ) \to \mathbb{C}$ is continuous. Suppose that the kernel
$$
(x,y) \in X \times X \to P(x,y):=\int_{\Omega}p_{w}(x,y)d\eta(w) \in \mathbb{C}
$$  
is well defined and continuous. Then, the kernel $P$ is positive definite and a  measure $\lambda \in \mathcal{M}_{\delta}(X)/\mathcal{M}_{c}(X)$ satisfy
$$
\int_{X}\int_{X}P(x,y)d\lambda(x)d\overline{\lambda}(y)=0
$$
if and only if
$$
\int_{X}\int_{X}p_{w}(x,y)d\lambda(x)d\overline{\lambda}(y)=0, \quad w \in Supp(\eta).
$$
If the kernel $P$ is bounded and the function $w \in \Omega \to \sup_{x \in X}p_{w}(x,x) \in \mathbb{R}$ is locally bounded, then  the same relation occur for $\lambda \in \mathcal{M}(X)$.
\end{lem}

\begin{proof}The fact that the kernel is positive definite is a consequence that each kernel $p_{w}$ is positive definite.\\
Now, let $\lambda \in \mathcal{M}_{c}(X)$, since the function $P$ is continuous it must be bounded on $Supp(\lambda)$.  By Fubini-Tonelli we can change the order of integration
$$
\int_{X}\int_{X}P(x,y)d\lambda(x)d\overline{\lambda}(y)= \int_{\Omega}\left [ \int_{X}\int_{X}p_{w}(x,y)d\lambda(x)d\overline{\lambda}(y) \right ]d\eta(w),
$$
 because $2|p_{w}(x,y)| \leq p_{w}(x,x) + p_{w}(y,y)$ and then $p \in L^{1}(\eta \times |\lambda| \times |\lambda|)$. The result we aim is a direct consequence that the function  
$$
w \in \Omega \to P_{\lambda}(w):=\int_{X}\int_{X}p_{w}(x,y)d\lambda(x)d\overline{\lambda}(y) \in \mathbb{R}
$$
 is continuous and nonnegative. Indeed, it is nonnegative because the kernel $p_{w}$ is positive definite. For the continuity, since $\{w\} \times Supp(\lambda)\times Supp(\lambda)$ is  a compact set   and $p$ is continuous, for every $\epsilon >0$ there exists  an open neighborhood $U_{w}$ of $w$ for which $|p_{w}(x,y) - p_{w^{\prime}}(x,y)| < \epsilon$ for all  $x, y \in Supp(\lambda)$  and $w^{\prime} \in U_{w}$. So $|P_{\lambda}(w)- P_{\lambda}(w^{\prime})| \leq \epsilon (|\lambda|(X))^{2}$ which proves our claim.\\
 If $P$ is bounded and the function $w \in \Omega \to \sup_{x \in X}p_{w}(x,x) \in \mathbb{R}$ is locally bounded, similar arguments can be used by replacing $Supp(\lambda)$ by a compact set $\mathcal{C}_{\epsilon}$ for which $|\lambda|(X) - |\lambda|(\mathcal{C}_{\epsilon}) < \epsilon$.\end{proof}

\subsection{Section \ref{Kernels on product of spaces}}

Recall the the formula

$$
	e^{-\|x\|^{2}/\sigma} = \frac{1}{2^{m}\pi^{m/2}} \sigma^{m/2}\int_{\mathbb{R}^{m}}e^{-ix \cdot \xi}
	e^{-\sigma\|\xi\|^{2}/4}d\xi, \quad x \in \mathbb{R}^{m} . 
	$$
Then
\begin{align*}
G_{A, \gamma}((u,x),(v,y))&=\frac{1}{2^{m}\pi^{m/2}} A(u,v)\gamma(u,v)^{m/2}\int_{\mathbb{R}^{m}}e^{-i(x-y) \cdot \xi}
	e^{-\gamma(u,v)\|\xi\|^{2}/4}d\xi\\
	&=\frac{1}{2^{m}\pi^{m/2}} C(u,v)\int_{\mathbb{R}^{m}}e^{-i(x-y) \cdot \xi}
	e^{-\gamma(u,v)\|\xi\|^{2}/4}d\xi
\end{align*}

\begin{proof}[\textbf{Proof of Theorem \ref{genGaussianuniversal}}]

The continuity of the kernel follows by its definition. It is positive definite because by the hypothesis, the kernel $C$ is positive definite and since the kernel  $\gamma$ is conditionally negative definite, the kernel $e^{-i(x-y) \cdot \xi}
	e^{-\gamma(u,v)\|\xi\|^{2}/4}$ is positive definite for every $\xi \in \mathbb{R}$. \\
If the kernel $G_{A, \gamma}$ is universal then it is SPD by definition. If the kernel $G_{A, \gamma}$ is SPD, then  for every $u \in X$ we have that $G_{A, \gamma}((u,0), (u,0))= A(u,u)>0$.\\ 
Now suppose that  $A(u,u)>0$ for every $u \in X$ . A measure  $\lambda \in \mathcal{M}_{c}(X\times \mathbb{R}^{m})$ is such that
$$
\int_{X\times \mathbb{R}^{m}}\int_{X\times \mathbb{R}^{m}}G_{A,\gamma}((u,x),(v,y))d\lambda(u,x)d\overline{\lambda}(v,y)=0 
$$
if and only if
\begin{equation}\label{genGaussianuniversaleq1}
    \int_{X\times \mathbb{R}^{m}}\int_{X\times \mathbb{R}^{m}}C(u,v)e^{-\|\xi\|^{2}\gamma(u,v)/4}e^{-i(x-y)\xi}d\lambda(u,x)d\overline{\lambda}(v,y)=0, \quad \xi \in \mathbb{R}^{m}.
\end{equation}
by Lemma \ref{Hadamardintegral}. When $\xi \neq 0$, by the hypothesis on the kernel $\gamma$, Theorem \ref{improvSchoen} and Lemma \ref{Schurmulgen1}  implies that  the kernel 
$$
(u,v) \in X \times X \to C(u,v)e^{-\|\xi\|^{2}\gamma(u,v)/4} \in \mathbb{C}
$$ 
is universal.  By Lemma \ref{Kroengen1}, we obtain that for every $A \in \mathscr{B}(X)$ it holds that
$$
0=\int_{ \mathbb{R}^{m}}\int_{ \mathbb{R}^{m}}e^{-i(x-y)\xi}d\lambda_{A}(x)d\overline{\lambda_{A}}(y) = |\widehat{\lambda_{A}}(\xi)|^{2}.
$$
for every $\xi \in \mathbb{R}^{m}\setminus{\{0\}}$. Since the only finite measure on $\mathcal{M}(\mathbb{R}^{m})$ that satisfies this relation is the zero measure, we must have that $\lambda(E\times B)=0$ for every $B \in \mathscr{B}(X)$ and $E \in \mathscr{B}(\mathbb{R}^{m})$, which implies that $\lambda$ is the zero measure and that the kernel $G_{A, \gamma}$ is universal.\\
\end{proof}

\begin{proof}[\textbf{Proof of Theorem \ref{genGaussianintegralllymetric}  }]
If $G_{A, \gamma}$ is ISPD, then  $A(u,u) >0$ for every $u \in X$ by  Theorem \ref{genGaussianuniversal}. \\
Conversely, since $|G_{A, \gamma}|\leq \sup_{u \in X} A(u,u)< \infty$,
$$
G_{A, \gamma}((u,x),(v,y))
=\frac{1}{2^{m}\pi^{m/2}} \int_{\mathbb{R}^{m}}C(u,v)e^{-i(x-y) \cdot \xi}e^{-\gamma(u,v)\|\xi\|^{2}/4}d\xi,
$$
 and $C$ is a bounded kernel, by Lemma \ref{Hadamardintegral} the kernel $G_{A, \gamma}$ is ISPD if and only if the only measure $\lambda  \in \mathcal{M}(X \times \mathbb{R}^{m})$ for which
$$
\int_{X\times \mathbb{R}^{m}}\int_{X\times \mathbb{R}^{m}}C(u,v)e^{-i(x-y) \cdot \xi}e^{-\gamma(u,v)\|\xi\|^{2}/4}d\lambda(u,x)d\overline{\lambda}(v,y)=0, \quad \xi \in \mathbb{R}^{m}
$$
is the zero measure. But, since $G_{t\gamma}$ is ISPD for every $t>0$,  Lemma \ref{Kroengen1}  together with the hypothesis that $C(u,u)>0$ for every $u \in X$ implies that 
$$
0=\int_{\mathbb{R}^{m}}\int_{ \mathbb{R}^{m}}e^{-i(x-y) \cdot \xi}d\lambda_{B}(x)d\overline{\lambda_{B}}(y)=|\widehat{\lambda_{B}}(\xi)|, \quad \xi \in \mathbb{R}^{m}\setminus{\{0\}}, \quad B \in \mathscr{B}(X). 
$$
Similar to the proof of Theorem \ref{genGaussianuniversal}, since the only measure on $\mathcal{M}(\mathbb{R}^{m})$ that satisfies this relation is the zero measure, we must have that $\lambda(E\times B)=0$ for every $B \in \mathscr{B}(X)$ and $E \in \mathscr{B}(\mathbb{R}^{m})$, which implies that $\lambda$ is the zero measure and that the kernel $G_{A, \gamma}$ is ISPD.\\
Now, we focus on the second relation. If $\mathcal{H}_{G_{A, \gamma}}\subset C_{0}(X\times \mathbb{R}^{m})$, then  by Proposition \ref{rkhscontained} for every $u \in X$ and $x \in \mathbb{R}^{m}$,   $G_{A, \gamma}((u,x)(u,x))= A(u,u)$ is a bounded function and $(G_{A, \gamma})_{(u,x)} \in C_{0}(X\times \mathbb{R}^{m})$, but then $A_{u}=(G_{A, \gamma})_{(u,x)}( \cdot, x)$ belongs to $C_{0}(X)$. By using Proposition \ref{rkhscontained} once again we have that $\mathcal{H}_{A} \subset C_{0}(X)$. \\
Conversely, suppose that $\mathcal{H}_{A} \subset C_{0}(X)$ and that $A$ is bounded by $1$, then for every $v \in X$ and  $\epsilon >0$ there exists a compact set $\mathcal{C}_{1,v} \subset X$ for which $|A(u,v)|< \epsilon$ for $u \in X \setminus \mathcal{C}_{1,v}$. Let $M = \sup_{z \in \mathcal{C}_{1,v}\cup\{v\}} \gamma(z,z)$, then $e^{-\|\xi\|^{2}/\gamma(u,v)}\leq e^{-\|\xi\|^{2}/M}$ for every $u \in \mathcal{C}_{1,v}$ and $\xi \in \mathbb{R}^{m}$, so if $y \in \mathbb{R}^{m}$ and  $\mathcal{C}_{2,y}$ is a compact set for which $e^{-\|x-y\|^{2}/M}< \epsilon$ for every $x \in \mathbb{R}^{m}\setminus \mathcal{C}_{2,y}$, the compact set $\mathcal{C} :=\mathcal{C}_{1,v}\times \mathcal{C}_{2,y}$ is such that
$|G_{A, \gamma}((u,x),(v,y))|< \epsilon$ for every $(u,x) \in X \times \mathbb{R}^{m} \setminus \mathcal{C}$, which concludes the argument.  \end{proof}

\begin{proof}[\textbf{Proof of Theorem \ref{secondCaponnetogeneralization1M22}}] Every  universal kernel  is strictly positive definite and  if the kernel $G_{A, \gamma}$ is strictly positive definite  then $G_{A, \gamma}(x,x)= A$ is a positive definite matrix. It only remains to prove that if the matrix $A$ is positive definite then the kernel $G_{A, \gamma}$ is universal. \\
Let  $\lambda_{1}, \ldots, \lambda_{\ell} \in \mathcal{M}_{c}(\mathbb{R}^{m})$ be complex valued  measures of compact support  for which
	$$
	\sum_{\mu, \nu=1}^{\ell}\int_{\mathbb{R}^{m}}\int_{\mathbb{R}^{m}}a_{\mu, \nu}e^{-\|x-y\|^{2}/\gamma_{\mu, \nu}}d\lambda_{\mu}(x) d\overline{\lambda_{\nu}}(y)=0. 
	$$
After a change on the order of integration  we can rewrite the previous equality as 
	$$
	\int_{\mathbb{R}^{m}}\sum_{\mu, \nu=1}^{\ell}c_{\mu, \nu}e^{-\gamma_{\mu, \nu}\|\xi\|^{2}/4}\widehat{\lambda_{\mu}}(\xi) \widehat{\overline{\lambda_{\nu}}}(\xi)d\xi=0. 
	$$
	Note that $\sum_{\mu, \nu=1}^{\ell}c_{\mu, \nu}e^{-\gamma_{\mu, \nu}\|\xi\|^{2}/4}\widehat{\lambda_{\mu}}(\xi) \widehat{\overline{\lambda_{\nu}}}(\xi) \geq 0$ for every $\xi \in \mathbb{R}^{m}$ and this function is continuous on the variable $\xi$, consequently we have that
	\begin{equation}\label{secondCaponnetogeneralization1M22conta}
	\sum_{\mu, \nu=1}^{\ell}c_{\mu, \nu}e^{-\gamma_{\mu, \nu}\|\xi\|^{2}/4}\widehat{\lambda_{\mu}}(\xi) \widehat{\overline{\lambda_{\nu}}}(\xi) =0, \quad \xi \in \mathbb{R}^{m}.
	\end{equation}
	 Since $\Gamma$ is conditionally negative definite with positive  coefficients, by equation \ref{condequa}  there exists $z_{1}, \ldots , z_{\ell} \in \mathbb{R}^{\ell}$ and $\varsigma_{1}, \ldots , \varsigma_{\ell} >0$ for which $\gamma_{\mu, \nu}= \|z_{\mu}- z_{\nu}\|^{2} + \varsigma_{\mu} + \varsigma_{\nu}$. If the points $z_{1}, \ldots , z_{\ell}$  are distinct, the  matrix $[c_{\mu, \nu}e^{-\gamma_{\mu, \nu}\|\xi\|^{2} /4}]_{\mu, \nu=1}^{\ell}$ is positive definite  and on this case the relation at Equation \ref{secondCaponnetogeneralization1M22conta} holds if and only if $\lambda_{\mu}$ is the zero measure for every $\mu$.\\
	 More generally, for arbitrary points $z_{1}, \ldots , z_{\ell}$  (not necessarily distinct),  consider the equivalence classes $F_{1}, \ldots ,F_{l} \subset \{1, \ldots , \ell \}$ for which  $F_{i}\cap F_{j} =\emptyset$ when $i \neq j $, $\cup_{i=1}^{l}F_{i}=\{1, \ldots , \ell \} $ and $\mu , \nu  \in F_{i}$ if and only if  $z_{\mu}= z_{\nu}$.  On this case the matrix $[e^{-\|z_{i}- z_{j}\|^{2}\|\xi\|^{2} /4}]_{i, j=1}^{l}$ is invertible and  the relation at Equation \ref{secondCaponnetogeneralization1M22conta} holds if and only if 
	\begin{equation}\label{secondCaponnetogeneralization1M22conta2}
	\sum_{\mu, \nu \in F_{i}}c_{\mu, \nu}e^{-(\varsigma_{\mu}+ \varsigma_{\nu})\|\xi\|^{2}/4}\widehat{\lambda_{\mu}}(\xi) \widehat{\overline{\lambda_{\nu}}}(\xi ) =0, \quad \xi \in \mathbb{R}^{m}, \quad 1 \leq i\leq l. 
	\end{equation}
	In order to simplify the notation, assume without loss of generalization that $F_{1}= \{1, \ldots , \ell  \}$.  Since the matrix $C$ is positive semidefinite, Equation \ref{secondCaponnetogeneralization1M22conta2} is equivalent at
	\begin{equation}\label{secondCaponnetogeneralization1M22conta2.5}
	[\widehat{\overline{\lambda_{\nu}}}(\xi )e^{-\varsigma_{\nu} \|\xi\|^{2}/4}]\sum_{\mu=1}^{\ell}c_{\mu, \nu}e^{-\varsigma_{\mu}\|\xi\|^{2}/4}\widehat{\lambda_{\mu}}(\xi)  =0, \quad \xi \in \mathbb{R}^{m}, \quad 1 \leq  \nu \leq l. 
	\end{equation}
	From now on we continue the proof of $(i)$ as if there are nonzero measures $\lambda_{\mu}$  that satisfy equation \ref{secondCaponnetogeneralization1M22conta2.5}, and we will obtain a contradiction. We may suppose without loss of generalization that all measures $\lambda_{\mu}$ are nonzero. By the Schwartz's Paley-Wiener theorem we know that the function $\widehat{\overline{\lambda_{\nu}}}(\xi )e^{-\varsigma_{\nu} \|\xi\|^{2}/4}$ is nonzero on a dense open set of $\mathbb{R}^{m}$, hence equation \ref{secondCaponnetogeneralization1M22conta2.5} is equivalent at
	\begin{equation}\label{secondCaponnetogeneralization1M22conta2.75}
	\sum_{\mu=1}^{\ell}c_{\mu, \nu}e^{-\varsigma_{\mu}\|\xi\|^{2}/4}\widehat{\lambda_{\mu}}(\xi)  =0, \quad \xi \in \mathbb{R}^{m}, \quad 1 \leq  \nu \leq l. 
	\end{equation}
	From the Fourier transform of equation \ref{secondCaponnetogeneralization1M22conta2.75}  at a point $-y \in \mathbb{R}^{m}$, we obtain that
	$$
	\sum_{\mu=1}^{\ell}a_{\mu, \nu}\int_{\mathbb{R}^{m}}e^{-\|x-y\|^{2}/ \varsigma_{\mu}}d\lambda_{\mu}(x) =0, \quad y \in \mathbb{R}^{m}, \quad 1 \leq  \nu \leq l.
	$$
	But then we reach the contradiction that all measures $\lambda_{\mu}$ are zero because the matrix $A$ is positive definite and the universality of the Gaussian kernel $G_{1/\varsigma_{\mu}}$ on $\mathbb{R}^{m}$.\\
	The proof of $(ii)$ is similar to the proof of $(i)$, up to Equation \ref{secondCaponnetogeneralization1M22conta2}. If the matrix $C_{i}: =[c_{\mu, \nu}]_{\mu, \nu \in F_{i}}$ is positive definite is immediate that all measures $\lambda_{\mu} \in \mathcal{M}(\mathbb{R}^{m})$ are zero. If the matrix $C_{i}$ is not positive definite, consider $v \in \mathbb{C}^{|F_{i}|} \setminus{\{0\}}$ for which $C_{i}v=0$ and a non zero function  $\phi \in C^{\infty}_{c}(\mathbb{R}^{m})$. Define $\lambda_{\mu} \in \mathcal{S}(\mathbb{R}^{m})$ (the space of Schwartz functions on $\mathbb{R}^{m}$) for which $\widehat{\lambda_{\mu}}(\xi)=v_{\mu}e^{\varsigma_{\mu}\|\xi\|^{2}/4 }\phi(\xi)$. Then if $v_{\mu} \neq 0 $  the finite Radon measure defined by the function $\lambda_{\mu}$ does not have compact support and 
	$$
	\sum_{\mu, \nu \in F_{i}}c_{\mu, \nu}e^{-(\varsigma_{\mu}+ \varsigma_{\nu})\|\xi\|^{2}/4}\widehat{\lambda_{\mu}}(\xi) \widehat{\overline{\lambda_{\nu}}}(\xi )= \sum_{\mu, \nu \in F_{i}}c_{\mu, \nu}v_{\mu}\overline{v_{\nu}}|\phi(\xi)|^{2}=0,
	$$
	for every $\xi \in \mathbb{R}^{m}$. We also emphasize that the relation $2\gamma_{\mu, \nu}= \gamma_{\mu, \mu} + \gamma_{\nu,\nu}$ occurs if and only if $z_{\mu}= z_{\nu}$.\end{proof}

\begin{proof}[\textbf{Proof of Lemma \ref{interaction}}] It is sufficient to prove the case of measures with compact support.\\
Let $\lambda_{1}, \ldots, \lambda_{\ell} \in \mathcal{M}(\mathbb{R}^{m})$ and  $0< \varsigma_{1} < \ldots< \varsigma_{\ell} $, such that
$$
\sum_{\mu=1}^{\ell} \int_{\mathbb{R}^{m}}e^{-\|x-y\|^{2}/\varsigma_{\mu}}d\lambda_{\mu}(x)=0, y \in \mathbb{R}^{m}.
$$

After a change in the order of integration, we have that
$$
\sum_{\mu=1}^{\ell}\varsigma_{\mu}^{m/2} \int_{\mathbb{R}^{m}}e^{iy\xi}e^{-\varsigma_{\mu}\|\xi\|^{2}/4}\widehat{\lambda_{\mu}}(\xi)d\xi=0, \quad y \in \mathbb{R}^{m}.
$$
Since the function $\sum_{\mu=1}^{\ell}\varsigma_{\mu}^{m/2}e^{-\varsigma_{\mu}\|\xi\|^{2}/4}\widehat{\lambda_{\mu}}(\xi)$ is a Schwartz function, the previous equality implies that this is the zero function. Then
	$$
	\widehat{\lambda_{1}}(\xi)=-\sum_{\mu=2}^{\ell} \frac{\varsigma_{\mu}^{m/2}}{\varsigma_{1}^{m/2}}e^{-(\varsigma_{\mu}-\varsigma_{1})\|\xi\|^{2}/4}\widehat{\lambda_{\mu}}(\xi).
	$$
	Since the right hand side of the previous equation is a Schwartz function, we obtain that $\lambda_{1} \in C_{c}^{\infty}(\mathbb{R}^{m})$, because the measure defined by it has compact support. Taking the Fourier transform on a point $-y \in \mathbb{R}^{m}$ we obtain that
	$$
	\lambda_{1}(y) =- \frac{\varsigma_{\mu}^{m/2}}{\varsigma_{1}^{m/2}(\varsigma_{\mu}-\varsigma_{1})^{m/2}}\sum_{\mu =2}^{\ell} \int_{\mathbb{R}^{m}}e^{-\|x-y \|^{2}/ (\varsigma_{\mu}-\varsigma_{1})}d\lambda_{\mu}(x),  
	$$
	so $\lambda_{1} \in \sum_{\mu=2}^{\ell}\mathcal{H}_{G_{[1/(\varsigma_{\mu}- \varsigma_{1})]}}$ (because the measure $\lambda_{\mu}$ is finite). By Theorem $3$ in \cite{minh2010some} $ \mathcal{H}_{G_{[1/(\varsigma_{\mu}- \varsigma_{1})]}} \subset \mathcal{H}_{G_{[1/(\varsigma_{2}- \varsigma_{1})]}} $ for every $\mu$ and then $\lambda_{1} \in \mathcal{H}_{G_{[1/(\varsigma_{2}- \varsigma_{1})]}} $. By the description of $\mathcal{H}_{G_{[1/(\varsigma_{2}- \varsigma_{1})]}} $  given in Theorem $1$ in  \cite{minh2010some}, we must have that $e^{1/(\varsigma_{2}- \varsigma_{1})\|y\|^{2}}\lambda_{1}(y)$ admits a power series expansion in $\mathbb{R}^{m}$,  but the function $\lambda_{1}$ has compact support, this implies that $\lambda_{1}$ must be the zero function. The rest of the proof follows by an induction argument. \end{proof}

Next Lemma is focused on the analysis on the kernel $\gamma(u,v)^{-1}$ on a broader context.

\begin{lem}\label{gammakernel} Let  $\gamma:X \times X \to (0, \infty)$ be a  continuous conditionally negative definite kernel and $\nu_{1}, \ldots , \nu_{\ell} \in (0,\infty)$. Consider  the matrix valued kernel 
$$
(u,v) \in X \times X \to \left [\frac{\Gamma(\nu_{i} + \nu_{j})}{\gamma(u,v)^{\nu_{i} + \nu_{j}}}\right ]_{i,j=1}^{\ell} \in M_{\ell}(\mathbb{C})
$$
\begin{enumerate}
\item[(i)] The matrix valued kernel is SPD if and only if $$
\{(i,j,u,v),  \gamma(u, v) = \gamma(u, u)=\gamma(v, v) \text{ and } \nu_{i}= \nu_{j} \}= \{(i,i,u,u),  1 \leq i \leq \ell, u \in X \}.
$$
\item[(ii)]If the kernel $\gamma$ is metrizable, then the matrix valued kernel is universal if and only if the $\nu_{i}$ are distinct.
\item[(iii)]If the kernel $\gamma$ is metrizable, then  the matrix valued kernel is ISPD if and only if the $\nu_{i}$ are distinct and  $\inf_{u \in X}\gamma(u,u)>0$.
\end{enumerate}
\end{lem}
\begin{proof}Indeed, if the sets are not equal and $(i,j,u,v)$ belongs to the left hand set but not to  the right hand set, then the interpolation matrix of the kernel at the points $u,v$ ($\ell\times \ell$ matrix if $u=v$ and $2\ell\times 2\ell$ matrix if $u\neq v$) is not a positive definite matrix. Conversely, by the definition of the gamma function
$$
\frac{\Gamma(\nu_{i} + \nu_{j})}{\gamma(u,v)^{\nu_{i} + \nu_{j}}}=\int_{(0,\infty)}t^{\nu_{i}}t^{\nu_{j}}e^{-\gamma(u,v)t}dt.
$$
So, by Lemma \ref{Hadamardintegral}  the kernel is SPD if and only if the only for every finite quantity of distinct points $u_{1}, \ldots, u_{m}$ and scalars $c_{i,\mu} \in \mathbb{R}$ for which
$$
\sum_{i,j=1}^{\ell}\sum_{\mu, \eta=1}^{m} c_{i, \mu}c_{j, \eta}t^{\nu_{i}}t^{\nu_{j}}e^{-\gamma(u_{\mu},u_{\eta})t}, \quad t \in (0, \infty)
$$
then  all scalars $c_{i, \mu}$ are equal to zero.   By Equation \ref{condequa}, we can write  $\gamma(u_{\mu},u_{\eta})= f(u_{\mu}) + \|h(u_{\mu}) - h(u_{\eta})\|^{2} + f(u_{\eta}) $. Consider  the equivalence class $\mu \simeq \eta$ if $h(u_{\mu}) = h(u_{\eta})$,  which  separates the set $\{1, \ldots , m\}$ on a finite number of disjoint sets $F_{1}, \ldots , F_{m^{\prime}}$, $m^{\prime } \leq m$. Note that
\begin{align*}
0&=\sum_{i,j=1}^{\ell}\sum_{\mu, \eta=1}^{m} c_{i, \mu}c_{j, \eta}t^{\nu_{i}}t^{\nu_{j}}e^{-\gamma(u_{\mu},u_{\eta}a)t}=\sum_{i,j=1}^{\ell}\sum_{\mu, \eta=1}^{m} c_{i, \mu}c_{j, \eta}t^{\nu_{i}}e^{-f(u_{\mu})t}t^{\nu_{j}}e^{-f(u_{\eta})t}e^{-\|h(u_{\mu})-h(u_{\eta})\|^{2}t}\\
&=\sum_{a,b=1}^{m^{\prime}} \left (\sum_{i=1}^{\ell}\sum_{\mu \in F_{a}} c_{i, \mu}t^{\nu_{i}}e^{-f(u_{\mu})t}\right )\left (\sum_{j=1}^{\ell}\sum_{\eta \in F_{b}} c_{j, \eta}t^{\nu_{j}}e^{-f(u_{\eta})t} \right )e^{-\|h(x_{a}) - h(x_{b})\|^{2}t}
\end{align*}
where $x_{1}, \ldots , x_{m}$ are class representatives. By Theorem \ref{Gauszinho} we have that  \\ $\sum_{i=1}^{\ell}\sum_{\mu \in F_{a}} c_{i, \mu}t^{\nu_{i}}e^{-f(u_{\mu})t}=0$ for every $t>0$ and $1\leq a \leq m$. Without loss of generalization suppose that $m^{\prime}=m$ and note that the pairs $(\nu_{i}, f(u_{\mu}))$  are distinct by the hypothesis. If $Z:= \{(\nu_{i}, f(u_{\mu}))\}$,  $X_{1} := argmin\{f(u_{\mu}), 1\leq \mu \leq m\}$, $y_{1}:= \min\{f(\mu), 1\leq \mu \leq m\}$ then the set of numbers $\nu^{1}:=\{ (\nu, u ), \quad (\nu, u ) \in (\{\nu\}\times X_{1}) \cap Z     \}$  are such that
$$
0= \sum_{i=1}^{\ell}\sum_{\mu =1}^{m} c_{i, \mu}t^{\nu_{i}}e^{-(f(\mu)+x_{1})t}= \sum_{(\nu, u) \in \nu^{1} }c_{\nu, u}t^{\nu} + \sum_{i=1}^{\ell}\sum_{\mu =1\mid \mu \notin X_{1}}^{m} c_{i, \mu}t^{\nu_{i}}e^{-(f(u_{\mu})+y_{1})t}.
$$
The first sum is a function that either is zero or  diverges in module as $t \to \infty$, while the second sum is a function that goes to zero as $t$ goes to infinity, then we must have that each sum is the zero function on $(0, \infty)$. But, the function $\sum_{(\nu, u) \in \nu^{1} }c_{\nu, u}t^{\nu}$ being zero on $(0, \infty)$, either there are two equal exponents $\nu$, which does not occur by the hypothesis, or all coefficients are zero. By an induction argument all coefficients $c_{i, \mu}$ are zero and then the kernel is SPD.\\    
As for  relation $(ii)$, since $e^{-\gamma(u,v)t}$ is an universal kernel for every $t>0$, Lemma \ref{Hadamardintegral} and Lemma \ref{Kroengen1} implies that the matrix valued kernel is  universal if and only if the only scalars $c_{1}, \ldots ,c_{\ell} \in \mathbb{R}$ for which
$$
\sum_{i=1}^{\ell}c_{i}t^{\nu_{i}}=0, \quad t \in (0,\infty)
$$
are all equal to zero. The result is then a consequence that  the set of functions $\{t^{\nu_{1}}, \ldots , t^{\nu_{\ell}}\}$ are linearly independent if and only if the exponents  $\nu_{1}, \ldots , \nu_{\ell}$ are distinct.\\
Relation $(iii)$ follows by similar arguments as relation $(ii)$. The condition $\inf_{u \in X}\gamma(u,u)>0$ is equivalent to the matrix valued kernel being bounded.\end{proof}

\begin{proof}[\textbf{Proof of Theorem \ref{genmaternuniversal}}]The continuity follows by the continuity of the functions involved. By the integral representation of  $\mathscr{M}(\|x-y\|; \alpha, \nu)$ and the proof of Theorem $1$ in \cite{Porcumaterngaussfamily} we have that
\begin{align*}
C_{i,j}^{\mathscr{M}, \gamma}((u,x), (v,y))&=A_{i,j}(u,v) \int_{(0, \infty)}e^{-\|x-y\|^{2}t/\gamma(u,v)} \left (  \left(\frac{\alpha_{i,j}^{2}}{4}\right )^{\nu_{i,j}}\frac{t^{-1-\nu_{i,j}}}{\Gamma(\nu_{i,j})}e^{-\alpha_{i,j}^{2}/4t}\right )dt\\
&=C_{i,j}(u,v)\frac{1}{\gamma^{m/2}(u,v)}\int_{(0, \infty)}e^{-\|x-y\|^{2}t/ \gamma(u,v)}m_{i,j}(t)dt,
\end{align*}
where $m_{i,j}(t)= m_{i}(t)m_{j}(t)$, $m_{i}(t)= \frac{\alpha_{i}^{2\nu_{i}}}{2^{2\nu_{i}}\Gamma(2\nu_{i})^{1/2}}t^{-\nu_{i} -1/2}e^{-\alpha_{i}^{2}/8t}$. The positivity of the kernel follows by this integral representation together  with the hypothesis on the kernel $C$ and the fact that $G_{A, \gamma}$ is a positive definite kernel.  \\
If the matrix valued kernel $C^{A, \gamma}$ is SPD (universal) then $C_{i,i}^{\mathscr{M}, \gamma}((0,u),(0,u))= A_{i,i}(u,u) >0$ for every $1\leq i \leq \ell$ and $u \in X$. Also, if $i \neq j $ is such that $\nu_{i}= \nu_{j}$ and $\alpha_{i}= \alpha_{j}$ then the scalar valued kernels $C_{i,j}^{\mathscr{M}, \gamma}, C_{i,i}^{\mathscr{M}, \gamma}, C_{j,j}^{\mathscr{M}, \gamma}$ are all equal, which does not occur if $C^{A, \gamma}$ is a matrix valued SPD kernel.\\
In order to prove the converse we analyse the matrix valued kernel 
\begin{equation}\label{genmaternuniversaleq1}
((u,x),(v,y)) \in (  X \times \mathbb{R}^{m})^{2} \to \left [\frac{1}{\gamma^{m/2}(u,v)}\int_{(0, \infty)}e^{-\|x-y\|^{2}t/ \gamma(u,v)}m_{i,j}(t)dt \right ]_{i,j=1}^{\ell} \in M_{\ell}(\mathbb{C}).
\end{equation}
Theorem \ref{genGaussianuniversal}  implies that the kernel $\gamma(u,v)^{-m/2}e^{-\|x-y\|^{2}t/ \gamma(u,v)}$ defined on $ X\times \mathbb{R}^{m}$ is universal for every $t>0$. Lemma \ref{Hadamardintegral} and Lemma \ref{Kroengen1} implies that  the matrix valued kernel on Equation  \ref{genmaternuniversaleq1} is SPD (universal) if and only if the matrix (which is independent from $u_{0}$) 
\begin{align*}
\left [\int_{(0, \infty)}m_{i,j}(t)dt \right ]_{i,j=1}^{\ell}&= 
\left [\frac{A_{i,j}(u_{0}, u_{0}) \gamma^{m/2}(u_{0},u_{0})}{C_{i,j}(u_{0},u_{0})} \right ]_{i,j=1}^{\ell}\\
&= \left [\frac{\alpha_{i}^{\nu_{i}}}{2^{-\nu_{i}}\Gamma(2\nu_{i})^{1/2}}\frac{ \alpha_{j}^{\nu_{j}}}{2^{ - \nu_{j}}\Gamma(2\nu_{j})^{1/2}}\frac{\Gamma(\nu_{i}+\nu_{j})}{(\alpha_{i}+\alpha_{j})^{\nu_{i}+ \nu_{j}}}\right ]_{i,j=1}^{\ell}
\end{align*}
is positive definite, which is characterized on  Lemma \ref{gammakernel}.\\
Lemma \ref{Schurmulgen1} implies that the kernel $C_{i,j}^{\mathscr{M}, \gamma}$ is SPD (universal) when the  kernel on equation \ref{genmaternuniversaleq1} is SPD(universal) and   $C_{i,i}(u,u)>0$ (or equivalently,  $A_{i,i}(u,u)>0$) for every $1\leq i \leq \ell$ and $u \in X$.
\end{proof}

\begin{proof}[\textbf{Proof of Theorem \ref{genmaternc0universal}}] The fact that if $\mathcal{H}_{C^{A, \gamma}} \subset C_{0}(X \times \mathbb{R}^{m}, \mathbb{C}^{\ell})$ then  $\mathcal{H}_{A}\subset C_{0}(X, \mathbb{C}^{\ell})$ is similar to the one presented at Theorem \ref{genGaussianintegralllymetric}. Conversely, if $\mathcal{H}_{A}\subset C_{0}(X, \mathbb{C}^{\ell})$ and the kernel $A$ is bounded by $1$, then for every $v \in X$ and $\epsilon >0$ there exists a compact set $\mathcal{C}_{v,\epsilon}$ for which $|A_{i,j}(u,v)|< \epsilon $ for every $u \in X \setminus \mathcal{C}_{v,\epsilon}$ and $1\leq i,j \leq \ell$. If $M:=inf_{u \in \mathcal{C}_{v,\epsilon} \cup\{v\}} \gamma(u,u)^{1/2}$, then 
$$
|\mathscr{M}(\|x-y\|/ \gamma(u,v)^{1/2}; \alpha_{i,j}, \nu_{i,j} )|\leq  \mathscr{M}(\|x-y\|/ M; \alpha_{i,j}, \nu_{i,j} ), \quad u \in \mathcal{C}_{v,\epsilon}, 1\leq i,j \leq \ell.
$$
The $\ell^{2}$ functions on the right hand side of previous equation are in $C_{0}(\mathbb{R}^{m})$, so for every $y \in \mathbb{R}^{m}$ there exists a compact set $\mathcal{C}_{y, \epsilon}$ for which  
$$
|\mathscr{M}(\|x-y\|/M; \alpha_{i,j}, \nu_{i,j} )| < \epsilon,  \quad x  \in \mathbb{R}^{m} \setminus \mathcal{C}_{y, \epsilon}, 1\leq i,j \leq \ell, 
$$
and then
$$
|C^{A \gamma}_{i,j}((u,x),(v,y))|< \epsilon, \quad (u,x) \in X \times \mathbb{R}^{m} \setminus  \mathcal{C}_{v,\epsilon} \times \mathcal{C}_{y, \epsilon}, 1\leq i,j \leq \ell. 
$$
Now, we focus on the second relation. If $C^{A, \gamma}$ is  $C_{0}(X \times \mathbb{R}^{m}, \mathbb{C}^{\ell})$-universal, then $\mathcal{H}_{C^{A, \gamma}} \subset C_{0}(X \times \mathbb{R}^{m}, \mathbb{C}^{\ell})$ by definition and by the first part of the theorem we must have that $\mathcal{H}_{A} \subset C_{0}(X, \mathbb{C}^{\ell})$. Also, $A_{i,i}(u,u) >0$ for every $u \in X$ and $1\leq i \leq \ell$ by  Theorem \ref{genGaussianuniversal}. \\
In order to prove the converse we analyse the matrix valued kernel 
\begin{equation}\label{genmaternc0universaleq1}
((u,x),(y,v)) \in (  X \times \mathbb{R}^{m})^{2} \to \left [\frac{1}{\gamma^{m/2}(u,v)}\int_{(0, \infty)}e^{-\|x-y\|^{2}t/ \gamma(u,v)}m_{i,j}(t)dt \right ]_{i,j=1}^{\ell} \in M_{\ell}(\mathbb{C}).
\end{equation}
By the hypothesis and Lemma \ref{prop17caponettogen}, the kernel $\gamma(u,v)^{-m/2}e^{-\|x-y\|^{2}t/ \gamma(u,v)}$ defined on $X\times \mathbb{R}^{m}$ is ISPD for every $t>0$ ($\|x-y\|^{2}t= \|(\sqrt{t}x) - (\sqrt{t}y)\|^{2}$). Lemma \ref{Hadamardintegral} and Lemma \ref{Kroengen1} implies that  the matrix valued kernel on Equation  \ref{genmaternc0universaleq1} is ISPD  if and only if the matrix (which is independent from $u_{0}$) 
\begin{align*}
\left [\int_{(0, \infty)}m_{i,j}(t)dt \right ]_{i,j=1}^{\ell}&= 
\left [\frac{A_{i,j}(u_{0}, u_{0}) \gamma^{m/2}(u_{0},u_{0})}{C_{i,j}(u_{0},u_{0})} \right ]_{i,j=1}^{\ell}\\
&= \left [\frac{\alpha_{i}^{\nu_{i}}}{2^{-\nu_{i}}\Gamma(2\nu_{i})^{1/2}}\frac{ \alpha_{j}^{\nu_{j}}}{2^{ - \nu_{j}}\Gamma(2\nu_{j})^{1/2}}\frac{\Gamma(\nu_{i}+\nu_{j})}{(\alpha_{i}+\alpha_{j})^{\nu_{i}+ \nu_{j}}}\right ]_{i,j=1}^{\ell}
\end{align*}
is positive definite, which is characterized  on  Lemma \ref{gammakernel}. Lemma \ref{Schurmulgen1} implies that the kernel $C_{i,j}^{\mathscr{M}, \gamma}$ is ISPD  when the kernel on Equation \ref{genmaternc0universaleq1} is ISPD and $C_{i,i}(u,u)>0$ (or equivalently,  $A_{i,i}(u,u)>0$) for every $1\leq i \leq \ell$ and $u \in X$.
\end{proof}

\begin{proof}[\textbf{Proof of Theorem \ref{genmaterninfinitouniversal}}]$(i)$ The continuity follows by the continuity of the functions involved. By the integral representation of  $\mathscr{M}(\|x-y\|; r, \nu)$, we have that
\begin{align*}
[\mathscr{M}_{A,\gamma}]_{i,j}((u,x),(v,y))&=A_{i,j}(u,v) \int_{(0, \infty)}e^{-\|x-y\|^{2}t} \left (  \left(\frac{\gamma(u,v)}{4}\right )^{\nu_{i} + \nu_{j}}\frac{t^{-1-\nu_{i}-\nu_{j}}}{\Gamma(\nu_{i}+\nu_{j})}e^{-\gamma(u,v)/4t}\right )dt\\
&=C_{i,j}(u,v)\int_{(0, \infty)}e^{-\|x-y\|^{2}t}e^{-\gamma(u,v)/4t}t^{-1}(4t)^{-\nu_{i}}(4t)^{-\nu_{j}}dt,
\end{align*}
and the positivity of the kernel follows by this representation. \\
$(ii)$ If $A_{i,i}(u,u)$ is not a positive number for some $u \in X$ and $1\leq i \leq \ell$ or the numbers $\nu_{1}, \ldots, \nu_{\ell}$ are not distinct, it is immediate that the kernel is not SPD.\\
Conversely, since $C_{i,i}(u,u)>0$ for every $1\leq i \leq \ell$ and $u \in X$, by Lemma \ref{Schurmulgen1} in order to prove that the kernel is SPD/universal   is sufficient to prove that the kernel defined by the integral on $(0, \infty)$ is SPD/universal.\\ 
The Gaussian kernel $e^{-\|x-y\|^{2}t}$ and  the Schoenberg kernel $e^{-\gamma(u,v)/t}$ are universal  for every $t \in (0, \infty)$, so by Lemma \ref{Hadamardintegral} and Lemma \ref{Kroengen1} the kernel defined by the integral on $(0, \infty)$ is SPD/universal if and only if the  only scalars $c_{1}, \ldots, c_{n} \in \mathbb{R}$ for which $\sum_{i=1}^{\ell}c_{i}t^{-\nu_{i}}=0$ for every  $t >0$ are all equal to zero, which holds true because the numbers $\nu_{i}$ are distinct.\\
$(iii)$ We focus on the proof of the converse relation.  Since $C_{i,i}(u,u)>0$ for every $1\leq i \leq \ell$ and $u \in X$, by Lemma \ref{Schurmulgen1} in order to prove that the kernel is  ISPD is sufficient to prove that the matrix valued kernel
$$\int_{(0, \infty)}e^{-\|x-y\|^{2}t}e^{-\gamma(u,v)/4t}t^{-1}(4t)^{-\nu_{i}}(4t)^{-\nu_{j}}dt,
$$
is ISPD. The Gaussian kernel $e^{-\|x-y\|^{2}t}$ and  the Schoenberg kernel $e^{-\gamma(u,v)/4t}$ are ISPD  for every $t \in (0, \infty)$, so by Lemma \ref{Hadamardintegral} and Lemma \ref{Kroengen1} the kernel defined by the integral on $(0, \infty)$ is ISPD if and only if the  only scalars $c_{1}, \ldots, c_{n} \in \mathbb{R}$ for which $\sum_{i=1}^{\ell}c_{i}t^{-\nu_{i}}=0$ for every  $t >0$ are all equal to zero, which holds true because the numbers $\nu_{i}$ are distinct.\\
$(iv)$ If   $\mathcal{H}_{A} \subset C_{0}(X, \mathbb{C}^{\ell} )$ and the kernel $A$ is bounded by $1$, then for every $v \in X$ and $\epsilon >0$ there exists a compact set $\mathcal{C}_{v,\epsilon}$ for which $|A_{i,j}(u,v)|< \epsilon $ for every $u \in X \setminus \mathcal{C}_{v,\epsilon}$ and $1\leq i,j \leq \ell$. If $M_{1}:=\inf_{u \in \mathcal{C}_{v,\epsilon} \cup\{v\}} \gamma(u,u)$ and $M_{2}:=\sup_{u \in \mathcal{C}_{v,\epsilon} \cup\{v\}} \gamma(u,u)$, then 
$$
|\mathscr{M}(\|x-y\|; \gamma(u,v)^{1/2}, \alpha_{i} + \alpha_{j})|\leq  \left (\frac{M_{2}}{M_{1}} \right )^{\nu_{i} + \nu_{j}}\mathscr{M}(\|x-y\|; M_{1}, \alpha_{i} + \alpha_{j} ),
$$
for every $u \in \mathcal{C}_{v,\epsilon}$, $1\leq i,j \leq \ell$. The $\ell^{2}$ functions on the right hand side of previous equation are in $C_{0}(\mathbb{R}^{m})$, so for every $y \in \mathbb{R}^{m}$ there exists a compact set $\mathcal{C}_{y, \epsilon}$ for which  
$$
\left |\left (\frac{M_{2}}{M_{1}} \right )^{\nu_{i} + \nu_{j}}\mathscr{M}(\|x-y\|; M_{1}, \alpha_{i} + \alpha_{j} ) \right | < \epsilon, \quad x \in \mathcal{C}_{y,\epsilon}, 1\leq i,j \leq \ell, 
$$
and then
$$
\left |[\mathscr{M}_{A,\gamma}]_{i,j}((u,x),(v,y)) \right |< \epsilon, \quad (u,x) \in X \times \mathbb{R}^{m} \setminus \mathcal{C}_{v,\epsilon} \times \mathcal{C}_{y, \epsilon}, 1\leq i,j \leq \ell.$$
The proof that if $\mathcal{H}_{\mathscr{M}_{A,\gamma}} \subset C_{0}(X \times \mathbb{R}^{m}, \mathbb{C}^{\ell})$ then  $\mathcal{H}_{A}\subset C_{0}(X, \mathbb{C}^{\ell})$ is similar to the one presented at Theorem \ref{genGaussianintegralllymetric}.
\end{proof}

\subsection{Section \ref{Kernels and hyperbolic spaces}}
Even though the results in this section are a direct consequence of Section \ref{Universality of Schoenberg-Gaussian kernels}, as mentioned in Section \ref{Hyperbolic  and log-conditional  kernels}, we present a direct proof for Theorem \ref{infihypersecondstep}. We focus on the hyperbolic spaces $\mathbb{H}^{m}$ only  to simplify the notation since several summations over multi-indexes are needed.

\begin{proof}[\textbf{(Partial) Proof of Theorem \ref{infihypersecondstep}}]
First, note that the kernel is bounded and
\begin{align*}
 [(x,t_{x}),(y,t_{y})  ]_{\mathbb{H}^{m}}^{-r}&=(t_{x}t_{y} - \langle x, y\rangle )^{-r}=(\sqrt{1+\|x\|^{2}}\sqrt{1+\|y\|^{2}} - \langle x, y\rangle )^{-r}\\
&= (1+\|x\|^{2})^{-r/2} \left (1 -  \left \langle \frac{x}{\sqrt{1+\|x\|^{2}}}, \frac{y}{\sqrt{1+\|y\|^{2}}} \right \rangle \right )^{-r} (1+\|y\|^{2})^{-r/2}. 
\end{align*}
The previous equation implies that $(x, t_{x}) \in \mathbb{H}^{m} \to [(x,t_{x}),(y,t_{y})  ]_{\mathbb{H}^{m}}^{-r} \in \mathbb{R}$ belongs to $C_{0}(\mathbb{H}^{m})$ for every $(y, t_{y}) \in \mathbb{H}^{m}$ because 
$$
[(x,t_{x}),(y,t_{y})  ]_{\mathbb{H}^{m}}^{-r} \leq (1+\|x\|^{2})^{-r/2} \left (1 -   \frac{\|x\|\|y\|}{\sqrt{1+\|x\|^{2}}\sqrt{1+\|y\|^{2}}}  \right )^{-r} (1+\|y\|^{2})^{-r/2},
$$
Proposition \ref{rkhscontained} implies that $\mathcal{H}_{[\cdot,\cdot]} \subset C_{0}(\mathbb{H}^{m})$. By the homeomorphism  $z= (x,t_{x}) \in \mathbb{H}^{m} \to x \in \mathbb{R}^{m}$, the kernel  is $C_{0}$-universal if and only if the only measure $\lambda \in \mathcal{M}(\mathbb{R}^{m})$ for which
\begin{equation}\label{hypeq2}
\int_{\mathbb{R}^{m}}\int_{\mathbb{R}^{m}} (1+\|x\|^{2})^{-r/2} \left (1 -  \left \langle \frac{x}{\sqrt{1+\|x\|^{2}}}, \frac{y}{\sqrt{1+\|y\|^{2}}} \right \rangle \right )^{-r} (1+\|y\|^{2})^{-r/2} d\lambda(x) d\lambda(y)=0
\end{equation}
is the zero measure. Since $|\langle x/\sqrt{1+\|x\|^{2}}, y/\sqrt{1+\|y\|^{2}}\rangle | < 1$ for every $x, y \in \mathbb{R}^{m}$, by the Taylor series of the hypergeometric functions $s  \to (1 - s)^{-r}$, the following series is absolutely convergent for every $r \in \mathbb{R}$
$$
\left (1 -  \left \langle \frac{x}{\sqrt{1+\|x\|^{2}}}, \frac{y}{\sqrt{1+\|y\|^{2}}} \right \rangle \right )^{-r}= \sum_{k=0}^{\infty}\binom{-r}{k}(-1)^{k}\left \langle \frac{x}{\sqrt{1+\|x\|^{2}}}, \frac{y}{\sqrt{1+\|y\|^{2}}} \right \rangle^{k},
$$
where $\binom{a}{0}=1$, $\binom{a}{1}=a$ and $\binom{a}{k+1}= \frac{(a-k)}{k}\binom{a}{k}$ , for every $a \in \mathbb{R}$. So, if a finite measure $\lambda \in \mathcal{M}(\mathbb{R}^{m})$ is such that Equation \ref{hypeq2} holds, then
$$
\int_{\mathbb{R}^{m}}\int_{\mathbb{R}^{m}} (1+\|x\|^{2})^{-r/2} \left \langle \frac{x}{\sqrt{1+\|x\|^{2}}}, \frac{y}{\sqrt{1+\|y\|^{2}}} \right \rangle^{k} (1+\|y\|^{2})^{-r/2} d\lambda(x) =0
$$
for every $k \in \mathbb{Z}_{+}$, and consequently
\begin{equation}\label{hypeq3}
\int_{\mathbb{R}^{m}} x^{\alpha}(1+\|x\|^{2})^{-r/2- |\alpha|/2}  d\lambda(x) =0
\end{equation}
for every $\alpha \in \mathbb{Z}_{+}^{m}$. We claim that
\begin{equation}\label{hypeq4}
\int_{\mathbb{R}^{m}} x^{\alpha}(1+\|x\|^{2})^{-v- |\alpha|/2}  d\lambda(x) =0
\end{equation}
for every $\alpha \in \mathbb{Z}_{+}^{m}$ and $v>0$. To prove this relation we follow a similar path as the one we made at Equation \ref{gausssunboundc0eq1}  on the proof of Theorem \ref{Gauszao}. We already know that it holds for every  $\alpha \in \mathbb{Z}_{+}^{m}$ and $v=r/2$, our induction step is to prove that  if it holds for every $\alpha \in \mathbb{Z}_{+}^{m}$ and a  $u>0$, then it holds for every $\alpha \in \mathbb{Z}_{+}^{m}$ and  $v \in (0, 2u)$. First, note that our induction hypothesis implies that
$$
\int_{\mathbb{R}^{m}} x^{\alpha}(1+\|x\|^{2})^{-u- |\alpha|/2} \left (\frac{\|x\|^{2}}{1 + \|x\|^{2}}  \right )^{k} d\lambda(x) =0
$$
for every $\alpha \in \mathbb{Z}_{+}^{m}$ and $k \in \mathbb{Z}_{+}$. By the Taylor series  expansion  of the hypergeometric function $s  \to (1 - s)^{v-u} \in \mathbb{R}$, the following series is absolutely convergent for every $x \in \mathbb{R}^{m}$
$$
(1 + \|x\|^{2})^{u-v}= \left (1 - \frac{\|x\|^{2}}{1 + \|x\|^{2}}  \right )^{v-u}= \sum_{k=0}^{\infty}(-1)^{k}\binom{v-u}{k}\left ( \frac{\|x\|^{2}}{1 + \|x\|^{2}} \right )^{k},
$$
moreover the function $x \in \mathbb{R}^{m} \to x^{\alpha}(1+\|x\|^{2})^{-v-|\alpha|/2} \in \mathbb{R}$ is bounded and also the function 
$$
 h(x):=\sum_{k=0}^{\infty} \left |x^{\alpha}(1+\|x\|^{2})^{-u-|\alpha|/2}(-1)^{k}\binom{v-u}{k}\left ( \frac{\|x\|^{2}}{1 + \|x\|^{2}} \right )^{k} \right |.   
$$
We separate the proof that the function $h$ is bounded in two cases.\\
\textbf{Case $1$:} When $v<u$, the function $h$ is bounded because $(-1)^{k}\binom{v-u}{k} \geq 0$ for every $k \in \mathbb{Z}_{+}$,  $|x^{\alpha}(1+\|x\|^{2})^{-|\alpha|/2}| \leq 1$ for every $x \in \mathbb{R}^{m}$ and  with these inequalities we obtain that $h(x) \leq (1+\|x\|^{2})^{-u}(1+\|x\|^{2})^{u-v}= (1+\|x\|^{2})^{-v}$.\\
\textbf{Case $2$:} When $v>u$, let $k_{0} \in \mathbb{Z}_{+}$ be such  such that $v-u-k_{0} \geq 0 $  but $v-u-k_{0}-1 < 0 $, then   $(-1)^{k_{0}}(-1)^{k}\binom{v-u}{k} \geq 0$ for every $k > k_{0}$, so if $h_{k}(x):= (1+\|x\|^{2})^{-u}(-1)^{k}\binom{v-u}{k}\left ( \frac{\|x\|^{2}}{1 + \|x\|^{2}} \right )^{k} $, then
\begin{align*}
h(x) \leq &\sum_{k=0}^{\infty} \left |h_{k}(x)\right |=\sum_{k=0}^{k_{0}}|h_{k}(x)| + (-1)^{k_{0}}\sum_{k_{0}+1}^{\infty}h_{k}(x)\\
&= \sum_{k=0}^{k_{0}}|h_{k}(x)| + (-1)^{k_{0}}\left [ (1+\|x\|^{2})^{-u}(1+ \|x\|^{2})^{u-v}  - \sum_{k=0}^{k_{0}}h_{k}(x) \right ]\\
&\leq 2\sum_{k=0}^{k_{0}}|h_{k}(x)| + (1+\|x\|^{2})^{-v}, 
\end{align*}
which proves that $h$ is a bounded function because each $h_{k}$ is a bounded function.\\
In particular, the dominated convergence theorem implies that
\begin{align*}
0&=\sum_{k=0}^{\infty}(-1)^{k}\binom{v-u}{k}\int_{\mathbb{R}^{m}} x^{\alpha}(1+\|x\|^{2})^{-u- |\alpha|/2} \left (\frac{\|x\|^{2}}{1 + \|x\|^{2}}  \right )^{k} d\lambda(x)\\
&= \int_{\mathbb{R}^{m}} x^{\alpha}(1+\|x\|^{2})^{-v- |\alpha|/2}  d\lambda(x)
\end{align*}
which settles the proof of  our claim. Now, consider the algebra of functions on $C_{0}(\mathbb{R}^{m})$
$$
\mathcal{A}:= span\{x \in \mathbb{R}^{m} \to x^{\alpha}(1+\|x\|^{2})^{-v-\alpha/2} \in \mathbb{R}, \quad  \alpha \in \mathbb{Z}_{+}^{m}, \quad v >0   \}.
$$
The function $h(x)=(1+\|x\|^{2})^{-1-\alpha/2} \in \mathcal{A} $ is such that $h(x)>0$ for every $x \in \mathbb{R}^{m}$, also, the algebra $\mathcal{A}$ separates points because if $x^{\alpha}(1+\|x\|^{2})^{-1-\alpha/2}= y^{\alpha}(1+\|y\|^{2})^{-1-\alpha/2}$ for every $\alpha \in \mathbb{R}^{m}$, then we must have that $[(x,t_{x}),(y, t_{y})]_{\mathcal{H}_{m}} = 1$, which can only occur if $x=y$. By the Stone-Weierstrass Theorem the algebra of functions  $\mathcal{A}$ is dense on $C_{0}(\mathbb{R}^{m})$. Since our claim made at Equation \ref{hypeq4} implies that for every $h \in \mathcal{A}$, we have that   $\int_{\mathbb{R}^{m}}h(x)d\lambda(x)=0$, the measure $\lambda$ must be the zero measure, which concludes the proof.\end{proof}

\begin{proof}[\textbf{Proof of Theorem \ref{infihypersecondstep}    }] Since the kernel $x,y \in \mathbb{H}\times \mathbb{H} \to [x,y] \in [1, \infty)$ is hyperbolic, $\log [x,y]$ is a conditionally negative definite kernel. If $2\log [x,y] = \log [x,x] + \log [y,y]$, then $[x,y]=1$, which only occur when $x=y$ because $d_{\mathbb{H}}(x,y)=0 $, implying that $\log [x,y]$ is a metrizable kernel (note that the same property occurs on the kernel $r\log [x,y]$, for $r>0$). Also, the function $\log [x,x]$ is a constant function,   Theorem \ref{integrallygaussianonmetric}  implies that the kernel
$$
(x,y) \in \mathbb{H} \times \mathbb{H} \to e^{-r\log [x,y]}= [x,y]^{-r}=\sech(d(x,y))^{r}  \in \mathbb{R}
$$
is ISPD for every $r>0$.
\end{proof}

\begin{proof}[\textbf{ Proof of Lemma \ref{hypGausskerinfty}}]
We only prove the case $r=1$, the others follows by a simple change of notation.\\
Suppose that $(i)$ holds, in particular, the function $e^{-\log [z, \xi_{0}]/2}$ also belongs to $C_{0}(X)$.  Since by Equation \ref{condequa} $\log[ z, w] = \|h(z) - h(w)\|^{2}$, by the parallelogram law, we have that  
$$
e^{-\|h(z) - h(\xi)\|^{2}}= e^{-\|h(z)-h(\xi_{0})\|^{2}/2 -  \|h(z)+h(\xi_{0}) - 2h(\xi)\|^{2}/2 + \|h(\xi)-h(\xi_{0})\|^{2}} \leq e^{\|h(\xi)-h(\xi_{0})\|^{2}}e^{-\log [z, \xi_{0}]/2},
$$
implying that the function $[z, \xi]^{-1}=e^{-\|h(z) - h(\xi)\|^{2}}$ belongs to $C_{0}(X)$ . It is immediate that $(ii)$ implies $(i)$\\ 
Relations $(ii)$ and $(iii)$ are equivalent by Proposition \ref{rkhscontained}.\\
If $(iv)$ holds, then for every $z \in X$ and $\epsilon >0$, the set $\{x \in X, \quad [x,z]^{-1}\geq \epsilon \}$ is bounded and closed on $X$, so it must be compact by the hypothesis implying that the function $H_{r, \xi} \in C_{0}(X)$ for every $\xi \in X$. For the converse, it is sufficient to show  that if $z \in X$ and $s>0$ then  the closed ball $d_{\mathbb{H}}(z, w)\leq s$ is compact, but this is  the set of points  that satisfies  $[z,w]^{-1} \geq (\cosh s )^{-1}$, which is compact by the $C_{0}$ hypothesis.\end{proof}


 Theorems \ref{hyperimprovSchoen}, \ref{integrallyhyperonmetric} and \ref{hyperkerc0infty} are a direct consequence of  the representation $L(x,y)^{-1}= e^{-\log L(x,y)}$ and Theorem \ref{improvSchoen}, \ref{integrallygaussianonmetric}  and \ref{gaussunboundc0} respectively, so we omit the proof.

\subsection{Dense algebras of bounded integrable functions on finite measures}\label{Dense algebras of bounded integrable functions on finite measures}

On this brief section we reprove a version of the  main result of \cite{Farrell}, but under the assumption that the functions involved are Borel measurable instead of Baire measurable and the measure is Radon and finite  instead of being $\sigma$-finite  and Baire. We also assume that the functions $h$ and $h_{i,j}$ are continuous and the sign properties holds everywhere on the set (on \cite{Farrell} it is only assumed the behaviour of its sign holds almost everywhere), this simple change simplifies the proof, but we remark that the continuity and the fact that the sign behaviour holds everywhere  is not a necessary hypothesis for the Theorem to hold.  

\begin{thm}\label{Farrelproof}Let $X$ be a  Hausdorff space and $\lambda \in \mathcal{M}(X)$ be a nonzero nonnegative measure. Let $\mathcal{A}$ be an algebra of  real valued functions in $L^{1}(\lambda)$ for which
\begin{enumerate}
    \item[(i)] Every $h \in \mathcal{A}$ belongs to $L^{\infty}(X)$.
    \item[(ii)] There exists  a continuous $h \in \mathcal{A}$ for which $h(x)>0$ for every $x \in X$.  
    \item[(iii)] There exists a basis $(U_{i})_{i \in \mathcal{I}}$ for the topology on $X$ such that if $U_{i}\cap U_{j}= \emptyset$ then for some continuous function $h_{i,j} \in \mathcal{A}$,  $h_{i,j}(x) > 0$ for $x \in U_{i}$ and $h_{i,j}(x) < 0$ for $x \in U_{j}$.
\end{enumerate}
Then the algebra $\mathcal{A}$ is dense on $L^{1}(\lambda)$
\end{thm}

\begin{proof}Relation $(i)$ ensures that products of functions in $\mathcal{A}$ are elements of $L^{1}(\lambda)$ by the Holder's inequality, which also implies that $\overline{\mathcal{A}}$ is an algebra on $L^{1}(\lambda)$\\
We show that $\mathcal{A}$ is dense in $L^{1}(\lambda)$ by showing that any continuous linear operator on $L^{1}(\lambda)$ that is zero on $\mathcal{A}$ is the zero operator. Indeed, let $I: L^{1}(\lambda) \to \mathbb{R}$ be a continuous operator that is zero on $\mathcal{A}$. Since $\lambda$ is finite there exists a function $\zeta \in L^{\infty}(\lambda)$ for which 
$$I(g)= \int_{X}g(x)\zeta(x)d\lambda(x)
=\int_{X}g(x)\zeta^{+}(x)d\lambda(x) - \int_{X}g(x)\zeta^{-}(x)d\lambda(x). 
$$
From this approach,  we can assume that $\mathcal{A}$ is a closed vector space. By a similar argument as the one in  Lemma $4.48$ in \cite{folland} page $140$, if $\phi, \psi \in \mathcal{A}$ then $\min(\psi, \phi)$, $\max(\psi, \phi)$, $\psi^{+}$ and $\psi^{-}$ belongs to $\mathcal{A}$. \\
We claim that the sets $X^{+}:=\{ x \in X , \quad \zeta(x)>0\}$ and $X^{-}:=\{ x \in X , \quad \zeta(x)<0\}$
 have $\lambda $ measure zero, which imply that $I$ is the zero functional. By the previous equality, it is sufficient to prove that $X^{+}$ has $\lambda$ measure zero.\\
 By the inner regularity of $\lambda$ on the sets  $X^{+}, X^{-}$, there exist two disjoint  sequences of nested compact sets $(\mathcal{C}_{+,n})_{n \in \mathbb{N}}$, $(\mathcal{C}_{-,n})_{n \in \mathbb{N}}$, for which
 $$
 \mathcal{C}_{+,n} \subset X^{+}, \lim_{n \to \infty}\lambda(\mathcal{C}_{+,n})= \lambda(X^{+}), \quad \mathcal{C}_{-,n} \subset X^{-}, \lim_{n \to \infty}\lambda(\mathcal{C}_{-,n})= \lambda(X^{-}).
 $$
Since $X$ is a Hausdorff space and the compact sets $\mathcal{C}_{+,n}$ and $\mathcal{C}_{-,n}$ are disjoint there exists disjoint open sets that separates them. Being the family of sets  $(U_{i})_{i \in \mathcal{I}}$ from relation $(iii)$ a basis for the topology on $X$, for every $n \in \mathbb{N}$ there exist  finite sets $F_{1,n}, F_{2,n} \subset \mathcal{I}$ such that   
$$
\mathcal{C}_{+,n}\subset \bigcup_{i \in F_{1,n}}U_{i}, \quad   \mathcal{C}_{-,n} \subset \bigcup_{j \in F_{2,n}}U_{j}   
$$ 
and $U_{i}\cap U_{j}= \emptyset $ if $i \in F_{1,n}$ and $j \in F_{2,n}$. The function $h_{n} := \max_{i \in F_{i,n}}(\min_{j \in F_{2,n}} h_{i,j}(x) ) \in \mathcal{A}$  is continuous,   $h_{n}(x) < 0$ on $\mathcal{C}_{-,n}$ and $h_{n}(x)>0$ on $\mathcal{C}_{+,n}$.\\ 
Let $a_{n}:= \min_{x \in \mathcal{C}_{+,n}}h_{n}(x)>0$, then  $g_{n}:=(\min(h_{n}/a_{n}, h))^{+} \in \mathcal{A}$,  and  $g_{n}(x) \geq \min (1, h(x))>0$ in $\mathcal{C}_{+,n}$ and $g_{n}(x)=0$ in $\mathcal{C}_{-,n}$. The function $k :=\inf_{n \in \mathbb{N}}g_{n}$ is well defined and is an element of  $\mathcal{A}$, because the infimum over the set $\{1,\ldots, m\}$ is a decreasing sequence of functions (bounded by $h$) in $\mathcal{A}$ and converges to $k$ in $L^{1}(\lambda)$ as $m$ goes to infinity by the dominated convergence theorem.\\
Note that $k>0$ almost everywhere on $X^{+}$ and $k=0$ almost everywhere on $X^{-}$, however
$$
0=I(k)=\int_{X}k(x)\zeta^{+}(x)d\lambda(x) - \int_{X}k(x)\zeta^{-}(x)d\lambda(x)=\int_{X}k(x)\zeta^{+}(x)d\lambda(x), 
$$
and $k\zeta^{+} \geq 0$, so $k\zeta^{+}$ is the zero function on $L^{1}(\lambda)$, which can only occur if $\lambda(X^{+})=0$.
\end{proof}

\bibliographystyle{siam}
\bibliography{Referrences.bib}

\end{document}